\documentclass[english]{article}
\usepackage{lmodern}
\usepackage[T1]{fontenc}
\usepackage[latin9]{inputenc}
\usepackage{geometry}
\usepackage{color}
\usepackage{babel}
\usepackage{units}
\usepackage{mathtools}
\usepackage{enumitem}
\usepackage{amsmath}
\usepackage{amsthm}
\usepackage{amssymb}
\PassOptionsToPackage{normalem}{ulem}
\usepackage{ulem}
\usepackage[unicode=true,pdfusetitle,
 bookmarks=false,
 breaklinks=false,pdfborder={0 0 0},pdfborderstyle={},backref=false,colorlinks=true]
 {hyperref}
\hypersetup{
 linkcolor=blue,  citecolor=blue, urlcolor=blue}

\makeatletter

\newcommand{\noun}[1]{\textsc{#1}}

\numberwithin{equation}{section}
\numberwithin{figure}{section}
\newlength{\lyxlabelwidth}      % auxiliary length 
  \theoremstyle{remark}
  \newtheorem*{acknowledgement*}{\protect\acknowledgementname}
\theoremstyle{plain}
\newtheorem{thm}{\protect\theoremname}[section]
  \theoremstyle{remark}
  \newtheorem{rem}[thm]{\protect\remarkname}
\newenvironment{elabeling}[2][]%
{\settowidth{\lyxlabelwidth}{#2}
\begin{description}[font=\normalfont,style=sameline,
leftmargin=\lyxlabelwidth,#1]}
{\end{description}}
  \theoremstyle{plain}
  \newtheorem{prop}[thm]{\protect\propositionname}
  \theoremstyle{plain}
  \newtheorem{cor}[thm]{\protect\corollaryname}

\usepackage{graphicx}
\usepackage{setspace}
\usepackage{url}
\usepackage[perpage,symbol]{footmisc}
\usepackage{multicol}
\usepackage{etoolbox}
\usepackage{bbm}
\usepackage{enumitem}
\usepackage{caption}
\usepackage[capbesideposition={left,center}]{floatrow}
\usepackage{etoolbox}

\date{}
\DefineFNsymbols*{lamportnostar}[math]{{(\dagger)}{(\ddagger)}{(\S)}{(\P)}{(\|)}{(\dagger\dagger)}{(\ddagger\ddagger)}}
\setfnsymbol{lamportnostar}

\DeclareMathOperator{\Spec}{Spec}
\DeclareMathOperator{\ord}{ord}

\DeclareMathOperator{\nr}{nr}

\newcommand{\one}{\mathbbm{1}}
\newcommand{\Mod}[1]{\,\left(\textup{mod}\;#1\right)}

\theoremstyle{plain}
\newtheorem{task}{Task}

\theoremstyle{definition}
\newtheorem*{rems*}{Remarks}
\newtheorem*{disc*}{Discussion}

\allowdisplaybreaks[2]

\renewenvironment{thebibliography}[1]{%
    \begin{oldthebibliography}{#1}%
      \setlength{\parskip}{0ex}%
      \setlength{\itemsep}{0ex}%
      \begin{small}
  }%
  {%
    \end{small}
    \end{oldthebibliography}%
  }

\makeatother

  \providecommand{\acknowledgementname}{Acknowledgement}
  \providecommand{\corollaryname}{Corollary}
  \providecommand{\propositionname}{Proposition}
  \providecommand{\remarkname}{Remark}
\providecommand{\theoremname}{Theorem}

\begin{document}

\title{Super-Golden-Gates for $PU\left(2\right)$}

\author{Ori Parzanchevski and Peter Sarnak}
\maketitle
\begin{center}
\emph{To David Kazhdan with admiration.}
\par\end{center}
\begin{abstract}
To each of the symmetry groups of the Platonic solids we adjoin a
carefully designed involution yielding topological generators of \emph{PU}(2)
which have optimal covering properties as well as efficient navigation.
These are a consequence of optimal strong approximation for integral
quadratic forms associated with certain special quaternion algebras
and their arithmetic groups. The generators give super efficient 1-qubit
quantum gates and are natural building blocks for the design of universal
quantum gates.
\end{abstract}

\section{Introduction}

The $n$-qubit circuits used for quantum computation are unitaries
in $U\left(\left(\mathbb{C}^{2}\right)^{\otimes n}\right)=U\left(2^{n}\right)$
which are products of elementary unitaries, each of which operates
on a fixed (typically at most $3$) number of qubits. The standard
universal gate set for quantum computing consists of all $1$-qubit
unitaries i.e.\ $U\left(2\right)$ (which can be applied to any of
the single qubits) and the $2$-bit XOR gate, together these generate
$U\left(2^{n}\right)$ (\cite{Nielsen2011QuantumComputationand}).
In a classical computer the only operations on a single bit are to
leave it or flip it. In the quantum setting we can rotate by these
$2\times2$ unitaries.

To further reduce to a finite universal gate set one has to settle
for a topologically dense set, and since overall phases do not matter
it suffices to find ``good'' topological generators of $G=PU\left(2\right)$.
The Solovay-Kitaev algorithm \cite{Nielsen2011QuantumComputationand}
ensures that any fixed topological generators of $G$ have reasonably
short words (i.e.\ circuits) to approximate any $x\in G$ (with respect
to the bi-invariant metric $d^{2}\left(x,y\right)=1-\frac{\left|\mathrm{trace}\left(x^{*}y\right)\right|}{2}$).
From the point of view of general polynomial type complexity classes
this is sufficient, however there is much interest (\cite{Nielsen2011QuantumComputationand,kliuchnikov2013fast,kliuchnikov2013asymptotically})
both theoretical and practical, to optimize the choice of such generators.

Golden-Gates (\cite{sarnak2015letter}) which correspond to special
arithmetic subgroups of unit quaternions and which were introduced
as optimal generating rotations in \cite{lubotzky1987hecke}, yield
variants of optimal generators. A particular case is the ``Clifford
plus T'' gates described below and which appear in most textbooks.
The Clifford gates form a finite subgroup $C_{24}$ of order $24$
in $G$ and to make the set universal one needs to add an extra element
of $G$. The popular choice is the order $8$ element $T=\left(\begin{smallmatrix}1 & 0\\
0 & e^{i\pi/4}
\end{smallmatrix}\right)$. That these generate an $S$-arithmetic group was shown in \cite{kliuchnikov2013fast},
see also \cite{sarnak2015letter}. 

Considerations of fault-tolerance when applying these to make circuits
in $U\left(2^{n}\right)$, require among other things that the universal
gate set consists of elements of finite order. Moreover for the Clifford
plus T gates, the applications of the $c$ gates with $c\in C_{24}$
in a circuit are considered to be of small cost compared to $T$ (\cite{Bravyi2005Universalquantumcomputation,Bocharov2012ResourceOptimalSingle}).
This leads to the ``T-count'' being the measure of complexity of
a word and to the problem that we consider in this note:

\emph{To find universal gate sets (i.e.\ ones that are topological
generators of $G$) which are of the form a finite group $C$ in $G$
together with an extra element $T$, which we take to be an involution,
so that $C$ plus $T$ is optimal with respect to covering $G$ with
a small $T$-count, and at the same time to be a able to navigate
$G$ efficiently with these gates.}

\bigskip{}

The key feature that was needed for a Golden-Gate construction was
that the corresponding $S$-arithmetic unit quaternion group act transitively
on the vertices of the corresponding $\left(q+1\right)$-regular tree
(here $q$ is a prime power). The extra ``miracle'' that is needed
here is that the group act transitively on the edges. With this extra
requirement there are only finitely many such ``Super-Golden-Gates'',
see Section \ref{sec:strong-approx}. We list some of them; in each
case the finite group $C$ is naturally a subgroup of the symmetries
of a platonic solid:
\begin{enumerate}
\item Pauli plus T \noun{(Cube})
\begin{align*}
C_{4} & =\left\langle \left(\begin{matrix}i & 0\\
0 & -i
\end{matrix}\right),\left(\begin{matrix}0 & 1\\
-1 & 0
\end{matrix}\right)\right\rangle ,\mbox{ the \ensuremath{4}-group of Pauli matrices}\\
T_{4} & =\left(\begin{matrix}1 & 1-i\\
1+i & -1
\end{matrix}\right).
\end{align*}
\item Minimal Clifford plus T \noun{(Octahedron})
\begin{align*}
C_{3} & =\left\{ \left(\begin{matrix}1 & 0\\
0 & 1
\end{matrix}\right),\left(\begin{matrix}1 & 1\\
i & -i
\end{matrix}\right),\left(\begin{matrix}1 & -i\\
1 & i
\end{matrix}\right)\right\} ,\mbox{ which is a subgroup of the Clifford group,}\\
T_{3} & =\left(\begin{matrix}0 & \sqrt{2}\\
1+i & 0
\end{matrix}\right).
\end{align*}
(Note: These generate a finite index subgroup of the usual ``Clifford
plus T'' group but the latter has redundancies in that the circuits
for the exactly synthesizable elements are not unique. For our choices,
it is.)
\item Hurwitz group plus T (\noun{Tetrahedron})
\begin{align*}
C_{12} & =\left\langle \left(\begin{matrix}i & 0\\
0 & -i
\end{matrix}\right),\left(\begin{matrix}1 & 1\\
i & -i
\end{matrix}\right)\right\rangle ,\\
T_{12} & =\left(\begin{matrix}3 & 1-i\\
1+i & -3
\end{matrix}\right).
\end{align*}
\item Clifford plus T \noun{(Octahedron})
\begin{align*}
C_{24} & =\left\langle \left(\begin{matrix}1 & 0\\
0 & i
\end{matrix}\right),\left(\begin{matrix}1 & 1\\
-1 & 1
\end{matrix}\right)\right\rangle ,\mbox{ the Clifford group}\\
T_{24} & =\left(\begin{matrix}-1-\sqrt{2} & 2-\sqrt{2}+i\\
2-\sqrt{2}-i & 1+\sqrt{2}
\end{matrix}\right)
\end{align*}
\item Klein's Icosahedral group plus T (\noun{Icosahedron)}
\begin{align*}
C_{60} & =\left\langle \left(\begin{matrix}1 & 1\\
i & -i
\end{matrix}\right),\left(\begin{matrix}1 & \varphi-i/\varphi\\
\varphi+i/\varphi & -1
\end{matrix}\right)\right\rangle ,\qquad\varphi=\frac{1+\sqrt{5}}{2};\\
T_{60} & =\left(\begin{matrix}2+\varphi & 1-i\\
1+i & -2-\varphi
\end{matrix}\right).
\end{align*}
\end{enumerate}
These Super-Golden-Gate sets all enjoy the same relative optimal distribution
and navigation properties that we describe next. The only difference
between them is the number $N\left(t\right)$ below, and from the
point of view of this count, example (5) is the best and presumably
it is the optimal absolute Super-Golden-Gate set.

Let $N\left(t\right)$ be the number of circuits in the elements of
$C$ and $T$, and of $T$-count $t$. That is words of the form $c_{0}Tc_{1}T\ldots Tc_{t}$
where the $c_{j}$'s are not $1$ except possibly at the ends. Clearly
\begin{equation}
N\left(t\right)=\left|C\right|^{2}\left(\left|C\right|-1\right)^{t-1},\qquad t\geq1\label{eq:N(t)}
\end{equation}
We want these $N\left(t\right)$ circuits to represent distinct elements
in $G$. In fact, for any $k$, we want the circuits of length at
most $k$ to realize distinct element in $G$ (these elements are
often referred to as the ``exactly synthesized'' gates). This requirement
is equivalent to the subgroup of $G$ generated by $C$ and $T$ being
isomorphic to $C*\left(\nicefrac{\mathbb{Z}}{2\mathbb{Z}}\right)$.
This is the case for the super-gates.

Next we want these $N\left(t\right)$ elements to almost cover $G$
optimally. That is almost all balls $B$ in $G$ (with respect to
the bi-invariant metric) of volume $V$, where $V\cdot N\gg\left(\log N\right)^{2}$,
should contain at least one of the $N=N\left(t\right)$ points. In
order to cover almost all balls, clearly $V\cdot N$ must be large,
so that for an almost all covering the above is essentially best possible,
that is almost all $x\in G$ have a circuit of essentially the shortest
possible T-count approximating $x.$ The Super-Golden-Gates enjoy
this optimal almost covering property (see Section \ref{sec:strong-approx}).
As with Golden-Gates (see \cite{sarnak2015letter}) these super-gates
do not cover all balls of this smallest size. There are rare balls
with volume $V\approx N^{-3/4}$ which are free of the $N=N\left(t\right)$
exactly synthesized elements, while every ball $B$ of volume at least
$N^{-1/2}$ does contain one of the $N\left(t\right)$ points. Whether
there exist gate sets which do not have this ``big hole'' feature
is an interesting open problem (\cite{sarnak2015letter,Rivin2017Inpreparation}). 

The final requirement for these super-gates is that we can find the
short circuits (whose existence is ensured from the discussion above)
efficiently. The task is given a ball $B$ in $G$ and a $k$, to
find (if it exists) a circuit in the gates of length at most $k$
and which lies in $B$. The problem is clearly in NP. The algorithm
introduced in {[}R-S{]} can be executed for our super-gate sets, and
it leads under the assumption that one can factor integers quickly,
to a heuristic algorithm, which for $B$ whose center is a \emph{diagonal}
matrix $y\in G$, resolves the above task in $\mathrm{poly}\left(\log\nicefrac{1}{V\left(B\right)}\right)$
steps. In particular for such a $y$ it finds for a given $k$, (one
of) the circuit of length at most $k$, which best approximates $y$.\footnote{Without the efficient factorization assumption one can find efficiently
a circuit in $B$ whose length is $\left(1+o\left(1\right)\right)$
times longer then the optimal one, see (\ref{eq:one-plus-o-one}).} On the other hand, the Diophantine problem (see below) that is at
the heart of the above algorithm for $y$ diagonal, is NP-complete
when $y$ is replaced by a general $x\in G$. Thus finding the shortest
circuit approximating a general $x$ in $G$ is apparently a genuinely
hard complexity problem. This does not preclude there being an efficient
algorithm which gives a good approximation to the shortest circuit.
By factoring the general $x$ in $G$ into a product $y_{1}y_{2}y_{3}$
where the $y_{j}$'s are in different diagonal subgroups (see Section
\ref{sec:strong-approx}) one can produce circuits with T-count which
is three times longer than the optimally short circuit. Removing this
factor of 3 remains a basic open problem concerning Golden Gates.

The Diophantine problem that underlies the analysis of these Golden
Gates is ``strong approximation for sums of four squares''. For
simplicity we restrict here to $\mathbb{Z}$ and $n\in\mathbb{N}$
odd (see Sections \ref{sec:Sums-of-squares} and \ref{sec:strong-approx}):
\begin{equation}
x_{1}^{2}+x_{2}^{2}+x_{3}^{2}+x_{4}^{2}=n.\label{eq:4-squares-intro}
\end{equation}
Let $S\left(n\right)$ be the set of integer solutions $x=\left(x_{1},x_{2},x_{3},x_{4}\right)$
to (\ref{eq:4-squares-intro}). To each $x\in S\left(n\right)$ let
$\widetilde{x}=\frac{x}{\sqrt{n}}\in S^{3}$, the unit sphere in $\mathbb{R}^{4}$
with its round metric and normalized volume $\mu$. In Section \ref{sec:strong-approx}
we show that these $\left|S\left(n\right)\right|$ points $\widetilde{x}$
almost cover $S^{3}$ optimally. More precisely, if $d\left(n\right)$
is the number of divisors of $n$ then if 
\[
\frac{V\cdot\left|S\left(n\right)\right|}{d\left(n\right)}\rightarrow\infty\quad\text{as}\quad n\rightarrow\infty,
\]
then 
\begin{equation}
\mu\left(S^{3}\backslash\left(\,\bigcup_{\smash{x\in S\left(n\right)}}B_{V}\left(\widetilde{x}\right)\right)\right)\rightarrow0.\label{eq:volume-cover}
\end{equation}
Here $B_{V}\left(\xi\right)$ is the ball in $S^{3}$ centered at
$\xi$ and of volume $V$. Among the ingredients in the proof of (\ref{eq:volume-cover})
are the Ramanujan Conjectures, which are theorems (\cite{deligne1974conjecture})
for the cases at hand. The computational complexity problem associated
with strong approximation for (\ref{eq:4-squares-intro}) is (see
Section \ref{sec:Sums-of-squares}):

\begin{task}\label{task:main-task}Given $n,\xi\in S^{3}$ and $V$
to find $x\in S\left(n\right)$ such that $\widetilde{x}\in B_{V}\left(\xi\right)$.\end{task}

This task is clearly in NP and Theorem \ref{thm:Task-IV-NP} shows
that it is NP-complete, at least under a radomized reduction. On the
other hand if $\xi$ is of the special form, $\xi=\left(\xi_{1},\xi_{2},0,0\right)$
then an adaption of \cite{ross2015optimal} is given in Section \ref{sec:Sums-of-squares}
which resolves the above task efficiently (i.e.\ in $\mathrm{poly}\left(\log n\right)$
steps). The algorithm assumes that one has an efficient algorithm
to factor $m$'s in $\mathbb{N}$, and for its running time it relies
on some heuristics (see Section \ref{subsec:Sums-of-four}).

We note that the analogous problem of an optimal topological generator
for $U\left(1\right)=\mathbb{R}/\mathbb{Z}$ has a golden solution.
If $R_{\alpha}$ is the rotation $x\mapsto x+\alpha$ (one could start
with two reflections whose composition is such a rotation) then the
covering volume $V\left(n,\alpha\right)$ if the words of length at
most $n$ (i.e.\ $S_{\alpha}\left(n\right)=\left\{ R_{\alpha}^{j};j=1,\ldots,n\right\} $)
is 
\[
V\left(n,\alpha\right)=\sup_{{I\cap S_{\alpha}\left(n\right)=\varnothing\atop I\text{ an interval}}}\left|I\right|.
\]
In \cite{graham1968distribution} it is shown that
\[
\overline{\lim\limits _{\mathclap{n\rightarrow\infty}}}\ V\left(n,\alpha\right)\left|S_{\alpha}\left(n\right)\right|\geq1+\frac{2}{\sqrt{5}}
\]
with equality iff $\alpha=\frac{a\varphi+b}{c\varphi+d}$ with $\varphi$
the golden ratio and $\left(\begin{smallmatrix}a & b\\
c & d
\end{smallmatrix}\right)\in GL_{2}\left(\mathbb{Z}\right)$. Moreover, the continued fraction algorithm allows one to find efficiently
the $j\leq n$ such that $R_{\varphi}^{j}$ best approximates a given
$\xi\in U\left(1\right)$.

We end the Introduction with a brief outline of the paper. In Section
\ref{sec:Sums-of-squares} the computational complexity results connected
with (\ref{eq:4-squares-intro}) and Task \ref{task:main-task} are
established. In Section \ref{subsec:Spectral-gap-and-SA} the optimal
covering properties of solutions to (\ref{eq:4-squares-intro}) and
its generalizations to number rings are proven. Section \ref{subsec:Quaternions-and-Ramanujan}
applies these results to unit groups of quaternions verifying the
advertised properties of Golden and Super-Golden gate sets. Section
\ref{sec:Super-Golden-Gates} is devoted to the construction of such
gate sets. Finally in Section \ref{sec:Iwahori-Hecke-operators} we
examine some semigroups and asymmetric random walks associated with
Super-Golden gates, and use the examples of Section \ref{sec:Super-Golden-Gates}
to construct Ramanujan Cayley digraphs.
\begin{acknowledgement*}
The authors thank O.\ Regev and J.\ Vrondák for for illuminating
discussions on integer programming and NP-completeness.
\end{acknowledgement*}

\section{\label{sec:Sums-of-squares}Sums of squares and complexity}

\subsection{\label{subsec:Sums-of-two}Sums of two squares}

The algorithm for navigating $G$ using Golden Gates is based on solving
the simplest quadratic Diophantine inequalities. We begin with the
setting of the integers $\mathbb{Z}$. $n=p_{1}^{e_{1}}\ldots p_{k}^{e_{k}}$
is a sum of two squares:
\begin{equation}
n=x^{2}+y^{2},\label{eq:sum-2-squares}
\end{equation}
iff each odd prime factor $p_{j}$ of $n$ with $e_{j}$ odd is congruent
to $1\Mod{4}$.

\begin{task}\label{task:2-squares}To solve (\ref{eq:sum-2-squares})
efficiently.\end{task}

By efficiently we mean in polynomial time in the input. In this case
the input is $n$ in $\mathbb{N}$ and specifying $n$ requires $\log n$
bits, so the Task is to find $x$ and $y$ satisfying (\ref{eq:sum-2-squares})
in $\left(\log n\right)^{c}$ steps, for some fixed $c$. More generally
denote by $h\left(\alpha\right)$ the height of a rational number
$\alpha=\frac{a}{b}$, that is $\max\left\{ \log\left|a\right|,\log\left|b\right|\right\} $,
so that $h\left(\alpha\right)$ measures the number of bits in $\alpha$.
As discussed in the introduction we assume throughout that factoring
an $n\in\mathbb{N}$ can be done efficiently. The following is due
to Schoof \cite{Schoof1985}.
\begin{thm}
Task \ref{task:2-squares} has an efficient resolution.
\end{thm}
\begin{proof}
First factor $n=2_{\phantom{1}}^{t_{\vphantom{0}}}p_{1}^{e_{1}}\ldots p_{k}^{e_{k}}$
(note $k$ and $e_{j}$ are $O\left(\log n\right)$) and for each
odd $e_{j}$ with $p\equiv1\Mod{4}$, solve 
\begin{equation}
x_{j}^{2}+y_{j}^{2}=p_{j}\label{eq:pj-2-squares}
\end{equation}
\cite{Schoof1985} gives a $O\left(\left(\log p_{j}\right)^{9}\right)$
algorithm to solve (\ref{eq:pj-2-squares}). One could also proceed
instead with various random (running time) algorithms for solving
(\ref{eq:pj-2-squares}) which run even faster and work very well
in practice \cite[§1.5.1]{Cohen1993courseincomputational}. With $x_{j}+iy_{j}\in\mathbb{Z}\left[\sqrt{-1}\right]$
at hand we can find a solution to (\ref{eq:sum-2-squares}) by taking
\[
x+iy:=\left(1+i\right)^{t}\prod_{e_{j}\text{ odd}}\left(x_{j}+iy_{j}\right)^{e_{j}}\prod_{e_{j}\text{ even}}p^{e_{j}/2},
\]
as $N\left(x+iy\right):=\left(x+iy\right)\left(x-iy\right)=n$. 

If we modify Task \ref{task:2-squares} by adding an inequality, the
complexity of the problem changes dramatically.
\end{proof}
\begin{task}\label{task:2-squares-ineq}Given $n\in\mathbb{N}$;
$\alpha,\beta\in\mathbb{Q}$ find a solution to (\ref{eq:sum-2-squares})
with $\alpha\leq\frac{y}{x}\leq\beta$.\end{task}

Task \ref{task:2-squares-ineq} is plainly in NP, that is to say one
can recognize a solution efficiently if one is presented with $x$
and $y$.
\begin{thm}
\label{thm:Task-II-NP}Task \ref{task:2-squares-ineq} is NP-complete
under a randomized reduction.
\end{thm}
What this says is that under a randomized reduction (see below) if
Task \ref{task:2-squares-ineq} has an efficient solution then P=NP.
So for us the upshot is that unlike factoring, Task \ref{task:2-squares-ineq}
is genuinely hard.
\begin{proof}
The Diophantine approximation condition in Task \ref{task:2-squares-ineq}
can be replaced by \ref{task:2-squares-ineq}': find $x+iy\in\mathbb{Z}\left[\sqrt{-1}\right]$,
$N\left(x+iy\right)=n$ and $\alpha'\leq\arg\left(x+iy\right)\leq\beta'$.
We show that if \ref{task:2-squares-ineq}' can be resolved efficiently
then so can the long known to by NP-complete subsum (or 'Knapsack'
as it is called in \cite{karp1972reducibility}) problem:

\begin{task}\label{task:knapsack}Given $t_{1},t_{2},\ldots,t_{n},t\in\mathbb{N}$,
are there $\varepsilon_{j}\in\left\{ 0,1\right\} $, $j=1,\ldots,k$
such that 
\[
\sum_{j=1}^{k}\varepsilon_{j}t_{j}=t.
\]
\end{task}

Note that the bit content in Task \ref{task:knapsack} is $H=\sum_{j=1}^{k}\log t_{j}$
and so we seek a $O\left(H^{c}\right)$ steps algorithm ($c$-fixed).
Let $M=\left\lfloor 100k\max_{j}\left\{ t_{j},t\right\} \right\rfloor $
and $\tau=\sum_{j=1}^{k}t_{j}$. We seek primes $P_{j}$ in the ring
$\mathbb{Z}\left[\sqrt{-1}\right]$ with $N\left(P_{j}\right)=X$
and $\arg\left(P_{j}\right)$ near $t_{j}/M$, $j=1,\ldots,k$, and
$X$ to be chosen. According to \cite{kubilius1955problem} the number
of primes $P$ with $N\left(P\right)\leq X$ and $\left|\arg\left(P\right)-\beta\right|\leq X^{-\frac{1}{10}}$
is asymptotic to 
\[
\frac{X}{\log X}\frac{X^{-\frac{1}{10}}}{\pi}\text{ as \ensuremath{X\rightarrow\infty}, and}
\]
 uniformly for any $\beta$. Hence for each $j=1,\ldots,k$ there
is $P_{j}$ with $N\left(P_{j}\right)\leq\left(100kM\right)^{10}$
and for which
\[
\left|\mathrm{Arg}\left(P_{j}\right)-\frac{t_{j}}{M}\right|\leq\frac{1}{100kM},
\]
(where $\mathrm{Arg}\left(z\right)\in\left(-\frac{\pi}{2},\frac{\pi}{2}\right]$).

We discuss how to find $P_{j}$ efficiently below, this leading to
the randomized reduction part. For $\varepsilon_{j}\in\left\{ 0,1\right\} $
and $\eta_{j}=2\varepsilon_{j}-1\in\left\{ -1,1\right\} $ we have
\begin{align*}
2\left[\sum_{j=1}^{k}\frac{\varepsilon_{j}t_{j}}{M}-\frac{t}{M}\right] & =\sum_{j=1}^{k}\eta_{j}\mathrm{Arg}\left(P_{j}\right)+\frac{\tau-2t}{M}+\sum_{j=1}^{k}\eta_{j}\left[\frac{t_{j}}{M}-\mathrm{Arg}\left(P_{j}\right)\right]\\
 & =\sum_{j=1}^{k}\eta_{j}\mathrm{Arg}\left(P_{j}\right)+\frac{\tau-2t}{M}+\left(\Theta\right)
\end{align*}
where 
\[
\left|\left(\Theta\right)\right|\leq\frac{1}{100M}.
\]
Hence 
\[
\sum_{j=1}^{k}\varepsilon_{j}t_{j}=1\quad\text{ iff }\quad\left|\sum_{j=1}^{k}\eta_{j}\mathrm{Arg}\left(P_{j}\right)-\frac{2t-\tau}{M}\right|\leq\frac{1}{100M}.
\]
Let 
\[
n=N\left(P_{1}\right)N\left(P_{2}\right)\ldots N\left(P_{k}\right)
\]
and 
\[
P_{j}=x_{j}+iy_{j}\text{ with }x_{j}>0.
\]
Then $z=x+iy$ solves Task \ref{task:2-squares-ineq}' with 
\[
\frac{2t-\tau}{M}-\frac{1}{100M}\leq\mathrm{Arg}z\leq\frac{2t-\tau}{M}+\frac{1}{100M}
\]
iff $z=P_{1}^{\sigma\left(\eta_{1}\right)}\ldots P_{k}^{\sigma\left(\eta_{k}\right)}$
and $\sum_{j=1}^{k}\varepsilon_{j}t_{j}=t$, where $\sigma\left(\eta_{j}\right)=\mathrm{id}$
if $\eta_{j}=1$ and $\sigma\left(\eta_{j}\right)$ is complex conjugation
if $\eta_{j}=-1$.

Note that $h\left(M\right)=O\left(H\right)$ and 
\[
h\left(n\right)\leq\sum_{j=1}^{k}h\left(N\left(P_{j}\right)\right)\leq k\log\left(100kM\right)^{10}\ll H^{2},
\]
so that the input for Task \ref{task:2-squares-ineq}' is polynomial
in terms of that of Task \ref{task:2-squares-ineq}. It remains to
find $P_{1},\ldots,P_{k}$ efficiently. That is to find a prime $P$
in the sector $N\left(P\right)\leq X$ and $\arg\left(P\right)\in S_{X}$
where $\left|S_{X}\right|=X^{-\frac{1}{10}}$ and $h\left(X\right)=H$.
If we choose a random $\beta\in\mathbb{Z}\left[\sqrt{-1}\right]$
in this section the probability that it is prime is $\nicefrac{1}{\log X}$.
Thus sampling such $\beta$'s and checking if they are prime will
produce the requisite $P$ in polynomial $H$ steps. In this way we
produce the $P_{j}$'s using a randomized polynomial time procedure.
Once the $P_{j}$'s are determined in polynomial $H$ steps then so
is $n$. Thus Task \ref{task:2-squares-ineq} is reduced to Task \ref{task:2-squares-ineq}'
albeit by a randomized reduction.
\end{proof}
\begin{rem}
A similar NP-complete problem is discussed in \cite{Sudan2010}. It
asserts that the task: Given $n,\alpha,\beta$ to find integers $x,y$
such that $xy=n$ and $\alpha\leq x\leq\beta$; is NP-complete under
a randomized reduction (see discussion \cite{np-complete-variant-of-factoring}).
\end{rem}

\subsection{\label{subsec:Sums-of-four}Sums of four squares}

The Diophantine approximation problem that is directly connected to
our navigation of $G$ is that of four squares,
\begin{equation}
x_{1}^{2}+x_{2}^{2}+x_{3}^{2}+x_{4}^{2}=n.\label{eq:4-squares}
\end{equation}
It is well known \cite{Lagrange1770Demonstrationduntheoreme} that
for every $n\geq0$ (\ref{eq:4-squares}) has integral solutions.
The question of finding a solution efficiently is discussed in \cite{Rabin1986Randomizedalgorithmsin}.
A randomized efficient algorithm to do this is to choose $x_{3}$
and $x_{4}$ with $x_{3}^{2}+x_{4}^{2}\leq n$ at random and to check
if $n-\left(x_{3}^{2}+x_{4}^{2}\right)$ is a prime $p\equiv1\Mod{4}$.
The last can be done efficiently (\cite{agrawal2004primes}) and if
yes, then Schoof gives $x_{1}$ and $x_{2}$ and we have a solution
to (\ref{eq:4-squares}). If $n-\left(x_{3}^{2}+x_{4}^{2}\right)$
is not such a prime we repeat with another choice of $x_{3},x_{4}$.
The probability of success is $\nicefrac{1}{\log n}$ in view of the
density of primes, and this leads to a randomized algorithm.

For the approximation problem we project the solutions $x=\left(x_{1},x_{2},x_{3},x_{4}\right)$
of (\ref{eq:4-squares}) onto $S^{3}=\left\{ \xi\in\mathbb{R}^{4}\,\middle|\,\left|\xi\right|=1\right\} $
by sending $x$ to
\begin{equation}
\widetilde{x}=\frac{x}{\sqrt{n}}.\label{eq:xtilde}
\end{equation}
$S^{3}$ comes with its round metric, relative to which all measurements
are made. So for example $B_{r}\left(\xi\right)$ is the ball centered
at $\xi$ of radius $r$.

\begin{task}\label{task:ball}Given $n,\xi,\varepsilon$ to find
a solution of (\ref{eq:4-squares}) with $\widetilde{x}\in B_{\varepsilon}\left(\xi\right)$.
\end{task}
\begin{rem}
In terms of the bit content or height $h$ we take $\varepsilon$
and $\xi$ to be rational and $h\left(\xi\right)=\max\left\{ h\left(\xi_{j}\right)\,\middle|\,j=1,2,3,4\right\} $.
\end{rem}
Task \ref{task:ball} is clearly in NP. We use Theorem \ref{thm:Task-II-NP}
to show:
\begin{thm}
\label{thm:Task-IV-NP}Task \ref{task:ball} is NP-complete under
a randomized reduction.
\end{thm}
\begin{proof}
While our interest is four squares, our proof of this is inductive.
We formulate Task \ref{task:ball} for a sum of $k$-squares. For
$k=2$ Theorem \ref{thm:Task-II-NP} asserts what is claimed here.
If Theorem \ref{thm:Task-IV-NP} is true for $k$ it is true for $k+1$.
We explain the case from 2 to 3, the general case is similar. Let
$n,\alpha',\beta'$ be the input for Task \ref{task:2-squares-ineq}'.
We show how to resolve it efficiently assuming Task \ref{task:ball}
for $x_{1}^{2}+x_{2}^{2}+x_{3}^{2}=m$ can be done efficiently. Choose
$\frac{n}{2}<y_{3}<n$ and let $m=y_{3}^{2}+n$, so $y_{3}^{2}<m<\left(y_{3}+1\right)^{2}$.
Hence for any $x=\left(x_{1},x_{2},x_{3}\right)$ with $x_{1}^{2}+x_{2}^{2}+x_{3}^{2}=m$
\[
d^{2}\left(\widetilde{x},\left(0,0,1\right)\right)=\frac{2\left(\sqrt{m}-x_{3}\right)}{m}\geq\frac{2\left(\sqrt{m}-y_{3}\right)}{m}
\]
with equality iff $x_{3}=y_{3}$. In particular those $\widetilde{x}$'s
which are all the closest solutions to $\left(0,0,1\right)$, correspond
exactly to solutions of
\[
x_{1}^{2}+x_{2}^{2}=n\quad\text{and}\quad x_{3}=y_{3}.
\]
So if we seek the solution $x=\left(x_{1,}x_{2}\right)\in S^{1}$
to (\ref{eq:sum-2-squares}) with $\frac{x}{\left|x\right|}$ closest
to $\alpha=\left(\alpha_{1},\alpha_{2}\right)\in S^{1}$, we can do
so by determining the $x=\left(x_{1},x_{2},x_{3}\right)$ with $x_{1}^{2}+x_{2}^{2}+x_{3}^{2}=m$
and for which $\widetilde{x}$ is closest to $\eta\left(\alpha_{1},\alpha_{2},0\right)+\left(0,0,1\right)$,
for $\eta$ small enough and positive. One checks that this determination
is efficient and hence it follows that an efficient solution of Task
\ref{task:ball} for $k=3$, yields one for $k=2$.
\end{proof}
While Theorem \ref{thm:Task-IV-NP} limits what one can do efficiently
in general, the good news is that special cases of this task can be
done. For $\xi=\left(\xi_{1},\xi_{2},0,0\right)$, Ross and Selinger
\cite{ross2015optimal} give an algorithm (see also \cite{Ross2015Optimalancillafree,blass2015optimal}).
\begin{thm}[Ross-Selinger]
\label{thm:ross-selinger}There is a heuristic efficient algorithm
to solve Task \ref{task:ball} for $\xi=\left(\xi_{1},\xi_{2},0,0\right)$.
\end{thm}
\begin{disc*}The precise meaning of heuristic will be clarified in
what follows. An important extra feature is that the algorithm can
and has, been implemented (\cite{ross2015optimal}) and that it runs
and terminates quickly. A similar algorithm has been devised and implemented
in the $p$-adic setting, namely that of navigating related Ramanujan
graphs in \cite{Petit2008FullcryptanalysisLPS} and \cite{Sardari2017Complexitystrongapproximation}.

The key property as far as these special $\xi$'s go is that for $x$
solving (\ref{eq:4-squares})
\[
d^{2}\left(\xi,\widetilde{x}\right)=2\left[1-\frac{\xi_{1}x_{1}+\xi_{2}x_{2}}{\sqrt{n}}\right],
\]
and this is a linear constraint depending on $x_{1}$ and $x_{2}$
only. Hence $\widetilde{x}\in B_{\varepsilon^{2}}\left(\xi\right)$
is equivalent to
\begin{elabeling}{00.00.0000}
\item [{$\left(A\right)$\hspace*{\fill}$\dfrac{\xi_{1}x_{1}+\xi_{2}x_{2}}{\sqrt{n}}>1-\dfrac{\varepsilon}{2}\ ;\quad x_{1}^{2}+x_{2}^{2}\leq n$\hspace*{\fill}}]~
\item [{and}]~
\item [{$\left(B\right)$\hspace*{\fill}$n-\left(x_{1}^{2}+x_{2}^{2}\right)=x_{3}^{2}+x_{4}^{2}.$\hspace*{\fill}}]~
\end{elabeling}
In this way tasks $\left(A\right)$ and $\left(B\right)$ are decoupled,
and we proceed by first finding candidate solution to $\left(A\right)$
and then solving for $\left(B\right)$. Now $\left(A\right)$ is a
problem of finding integer lattice points $\left(x_{1},x_{2}\right)$
in $\mathbb{Z}^{2}$ which lie in the convex 'miniscus' region defined
by the inequalities in $\left(A\right)$. Listing the solutions one
at a time (which is done with polynomial cost) is in P. The algorithm
to do this is due to Lenstra \cite{lenstra1983integer} and applies
to such integer convex programming problems for $\mathbb{Z}^{k}$
in any \emph{fixed} dimension $k$ (see \cite{Lovasz1986algorithmictheorynumbers}).
His key input being Minkowski reduction of bases of lattices. If the
input complexity for Task \ref{task:ball} is $H$, then we examine
$O\left(H^{c}\right)$of the solutions to $\left(A\right)$ (in some
geometric ordering or randomly) and for each $\left(x_{1},x_{2}\right)$
we check if $\left(B\right)$ has a solution by running the efficient
solution to Task (\ref{task:2-squares}). If we arrive at a solution
$\left(x_{1},x_{2},x_{3},x_{4}\right)$ then we have resolved Task
\ref{task:ball}. If our $O\left(H^{c}\right)$ (here $c$ is a fixed
number such as $10$) steps cover all the solutions to $\left(A\right)$
and no $\left(x_{3},x_{4}\right)$ is found, then we output that Task
\ref{task:ball} has no solution, which is the case. The only stumbling
block to the algorithm terminating efficiently is that there are more
than a polynomial in $H$ number of solutions to $\left(A\right)$
and our $O\left(H^{c}\right)$ inspections produce no solutions to
$\left(B\right)$, in particular this would happen if in fact there
are no solutions to $\left(B\right)$ among these very many solutions
to $\left(A\right)$. The heuristic aspect of the algorithm is that
this last scenario will not happen. The density of numbers $t$ in
$\left[X,2X\right]$ which are sums of two squares is $\nicefrac{1}{\sqrt{\log X}}$,
so that one expects that $\left(B\right)$ will have a solution with
probability $\nicefrac{1}{\sqrt{H}}$. Thus the probability of hitting
this stumbling block after $O\left(H^{c}\right)$ tries is very small,
and possibly never arises. This completes the analysis and meaning
of Theorem \ref{thm:ross-selinger}.\end{disc*}
\begin{rem}
\label{rem:without-factorization}In Theorem \ref{thm:ross-selinger}
and elsewhere we have assumed that one can factor efficiently. Without
appealing to such a factoring algorithm we can seek an approximation
to $\xi$ as above but in $\left(B\right)$ we require that $n-\left(x_{1}^{2}+x_{2}^{2}\right)$
be a prime $p\equiv1\left(4\right)$. The density of such special
solutions is a bit smaller ($\nicefrac{1}{\log X}$ rather than $\nicefrac{1}{\sqrt{\log X}}$)
and we can find these solutions efficiently.
\end{rem}

\subsection{Sums of squares in number rings}

The coordinates of the Golden Gates Sets lie in the ring of $S$-integers
of certain number fields. This leads us to the examination of the
tasks discussed in Sections \ref{subsec:Sums-of-two} and \ref{subsec:Sums-of-four}
with $\mathbb{Z}$ replaced by the ring of integers $\mathcal{O}_{K}$,
of number fields such as $K=\mathbb{Q}\left(\sqrt{2}\right)$ and
$K=\mathbb{Q}\left(\sqrt{5}\right)$. These are unique factorization
domains and in fact even Euclidean domains. This allows us to extend
the results of the previous sections to these rings. We explicate
what we need for later. Task \ref{task:2-squares} is to solve (\ref{eq:sum-2-squares})
with $x,y\in\mathcal{O}_{K}$ ($K$ fixed) given $n\in\mathcal{O}_{K}$.
By factoring $N\left(n\right)$ (have $N=N_{K/\mathbb{Q}}$ is the
norm) into primes and applying \cite{Schoof1985} to $N_{K/\mathbb{Q}}\left(\alpha\right)=p$,
we obtain an efficient factorization of $n$ into primes in $\mathcal{O}_{K}$.
Thus we are left with dealing with the case that $n=P$, is a totally
positive prime. Using the Euclidean algorithm to compute $\gcd$'s
in $\mathcal{O}_{K}$ we are in the same situation as over $\mathbb{Z}$,
needing to find $\nu\in\nicefrac{\mathcal{\mathcal{O}}}{P}$ with
$\nu^{2}+1\equiv0\Mod{P}$. Again we apply \cite{Schoof1985} and
this yields an efficient solution to Task \ref{task:2-squares} over
$\mathcal{O}_{K}$.

With this we can address Task \ref{task:ball} over $\mathcal{O}_{K}$
in the form of Theorem \ref{thm:ross-selinger}. That is to solve
\begin{equation}
x_{1}^{2}+x_{2}^{2}+x_{3}^{2}+x_{4}^{2}=n\quad\text{with}\quad\widetilde{x}\in B_{\varepsilon}\left(\xi\right),\xi=\left(\xi_{1},\xi_{2},0,0\right).\label{eq:4-squares-pole}
\end{equation}
Here $n$ is totally positive, i.e.\ $n>0$ and $n'>0$ where $'$
is the Galois conjugate of $\nicefrac{K}{\mathbb{Q}}$, and $\widetilde{x}=\frac{x}{\sqrt{n}}\in S^{3}$. 

As before this decouples into
\begin{elabeling}{00.00.0000}
\item [{$\left(A'\right)$\hspace*{\fill}$x_{1},x_{2}\in\mathcal{O}_{K},\quad\dfrac{\xi_{1}x_{1}+\xi_{2}x_{2}}{\sqrt{n}}>1-\frac{\varepsilon}{2}\ ;\quad\begin{matrix}x_{1}^{2}+x_{2}^{2}\leq n\vphantom{\Big|}\\
\left(x_{1}'\right)^{2}+\left(x_{2}'\right)^{2}\leq n'\vphantom{\Big|}
\end{matrix}$\hspace*{\fill}}]~
\item [{and}]~
\item [{$\left(B'\right)$\hspace*{\fill}$n-\left(x_{1}^{2}+x_{2}^{2}\right)=x_{3}^{2}+x_{4}^{2}\quad\text{with}\quad x_{3},x_{4}\in\mathcal{O}_{K}.$\hspace*{\fill}}]~
\end{elabeling}
For $\left(A'\right)$ the points $\left(x_{1},x_{1}',x_{2},x_{2}'\right)$
with $x_{1},x_{2}\in\mathcal{O}_{K}$ form a rank $4$ lattice $L$
in $\mathbb{R}^{4}$ and the inequalities defining the set in which
they lie is a convex (compact) set. Lenstra's algorithm allows us
to find these points efficiently. As in the setting with $\left(A\right),\left(B\right)$,
each instance of $\left(B'\right)$ with a given $x_{1},x_{2}\in\mathcal{O}_{K}$
is Task \ref{task:2-squares} over $\mathcal{O}_{K}$, and so has
an efficient solution. The rest of the analysis is as before. We have
shown that Theorem \ref{thm:ross-selinger} is valid for $\mathcal{O}_{K}$.

\section{\label{sec:strong-approx}Strong approximation for integer points
on spheres}

\subsection{\label{subsec:Spectral-gap-and-SA}Spectral gap and strong approximation}

We are interested in how well the points $\widetilde{x}$ for $x$
satisfying (\ref{eq:4-squares}) or (\ref{eq:4-squares-pole}) cover
$S^{3}$. A powerful method to address this makes use of these points
corresponding to orbits defining Hecke operators, and their eigenvalues
may be estimated (optimally) using the Ramanujan Conjectures (see
\cite{lubotzky1987hecke}).

We formulate a covering estimate in terms of a spectral gap, for a
general compact topological group $L$ as it clarifies the roles of
the ingredients (we will apply it for $G=L=PU\left(2\right)$). Let
$\mu$ be Haar measure on $L$ normalized to be a probability measure.
Let $S$ be a finite subset of $L$ and $\left|S\right|$ its cardinality.
For $B\subseteq L$ with $\mu\left(B\right)>0$, we are interested
in $\bigcup_{s\in S}\left(Bs\right)$. If these right $S$-translates
of $B$ cover all of $L$ we say that the pair $B,S$ covers. If $\mu\left(L\backslash\left(\bigcup_{s\in S}Bs\right)\right)=o\left(1\right)$
as $\left|S\right|\rightarrow\infty$ (our interest is $\left|S\right|\rightarrow\infty$
and we allow everything including $L$ to change) we say that $B,S$
is almost covering. It is clear that to be almost covering we must
have 
\[
\left|S\right|\mu\left(B\right)\geq1+o\left(1\right).
\]
If $S$ is chosen at random (that is i.i.d.\ w.r.t.\ $\mu$) then
\begin{equation}
\left|S\right|\mu\left(B\right)\rightarrow\infty\label{eq:random-almost-cover}
\end{equation}
is necessary and sufficient for almost covering\footnote{This being a well known property of the Coupon Collector's problem.}.

Let $a_{s}\geq0$ with $\sum_{s\in S}a_{s}=1$ and let $\nu_{s}$
be the probability measure on $L$ given by 
\[
\nu_{S,a}:=\sum_{s\in S}a_{s}\delta_{s}
\]
($\delta_{x}$ is a point mass at $x$). $\nu_{s}$ defines a right
convolution operator $T_{S,a}$ on $L^{2}\left(L\right)$ by 
\begin{equation}
T_{S,a}f\left(x\right)=\sum_{s\in S}a_{s}f\left(xs^{-1}\right).\label{eq:Hecke-operator}
\end{equation}
Clearly
\[
T_{S,a}\one=\one;\qquad\left\Vert T_{S,a}\right\Vert =1.
\]
The orthogonal space to $\one$, $L_{0}^{2}\left(L\right)$, is $T_{S,a}$
invariant. Let 
\[
W=W_{S,a}:=\left\Vert T_{S,a}\big|_{L_{0}^{2}\left(L\right)}\right\Vert .
\]
$W\leq1$ and we estimate the covering properties of $S$ in terms
of the spectral norm $W$.
\begin{prop}
\label{prop:spec-norm-covering}For $B$ and $S$ as above $\mu\left(L\backslash\left(\bigcup_{s\in S}Bs\right)\right)\leq\frac{W^{2}}{\mu\left(B\right)}$,
in particular if $\frac{W^{2}}{\mu\left(B\right)}=o\left(1\right)$,
then $B,S$ is almost covering.
\end{prop}
\begin{proof}
Let $I_{B}\left(x\right)$ be the indicator function of the set $B$.
$I_{B}\left(x\right)-\mu\left(B\right)$ is in $L_{0}^{2}\left(L\right)$
and $\int_{L}\left(I_{B}\left(x\right)-\mu\left(B\right)\right)^{2}d\mu\left(x\right)=\mu\left(B\right)\left(1-\mu\left(B\right)\right)$.
Hence $\int_{L}\left[T_{S,a}\left[I_{B}\left(x\right)-\mu\left(B\right)\right]\right]^{2}d\mu\left(x\right)\leq W^{2}\mu\left(B\right)\left(1-\mu\left(B\right)\right)$,
that is 
\begin{equation}
\int_{L}\left[\sum_{s\in S}a_{s}I_{B}\left(xs^{-1}\right)-\mu\left(B\right)\right]^{2}d\mu\left(x\right)\leq W^{2}\mu\left(B\right)\left(1-\mu\left(B\right)\right).\label{eq:variance-ineq}
\end{equation}
If $x\notin Bs$ for any $s\in S$ then $xs^{-1}\notin B$ and the
bracketed expression in (\ref{eq:variance-ineq}) is $\left[\mu\left(B\right)\right]^{2}$,
and hence $\mu\left(L\backslash\bigcup_{s\in S}Bs\right)\mu\left(B\right)^{2}\leq W^{2}\mu\left(B\right)$,
which gives the claimed inequality.
\end{proof}
If $L$ is a continuous group then we can choose $B$ to satisfy $\mu\left(B\right)=\frac{1}{2\left|S\right|}$,
so that $\mu\left(L\backslash\bigcup_{s\in S}Bs\right)\geq\frac{1}{2}$
and hence from Proposition \ref{prop:spec-norm-covering}, $W^{2}\geq\frac{1}{4\left|S\right|}$
or $W\geq\frac{1}{2\left|S\right|^{1/2}}$. In our applications $W$
will be of this size, namely $\nicefrac{1}{\sqrt{\left|S\right|}}$,
and in this case $\nicefrac{W^{2}}{\mu\left(B\right)}$ is $o\left(1\right)$
iff $\mu\left(B\right)\left|S\right|\rightarrow\infty$, matching
(\ref{eq:random-almost-cover}). Hence in such cases $B,S$ is essentially
optimal in its almost covering property.

To conclude that $B,S$ is actually covering we need to assume more
about the shape of $B$. Specifically we take $B$ to be a ball which
is the case of interest to us. Let $d\left(x,y\right)$ be a $L$-bi-invariant
metric on $L$ (if $L$ is finite then $d$ is the discrete metric
$d\left(x,y\right)=0$ if $x=y$, $d\left(x,y\right)=\infty$ if $x\neq y$),
and let $B$ be a ball centered at $e$ of volume $\mu\left(B\right)=v$.
Then $B=B^{-1}$ and $sB=Bs$ is the ball centered at $s$ and of
volume $v$. $B,S$ is covering iff every ball in $L$ of volume $v$
contains an $s\in S$.
\begin{cor}
\label{cor:ball-cover}If $B$ is a ball in $L$ with center $e$,
radius $r\geq0$ and volume $v>W$ then $B_{2r}\left(e\right)$, $S$
is covering (where $B_{r}\left(x\right)$ is the ball centered at
$x$ and of radius $r$). 
\end{cor}
\begin{proof}
If there is $x\in L$ such that no $s\in S$ meets $B_{2r}\left(x\right)$,
then for each $y\in B_{r}\left(x\right)$, $B_{r}\left(y\right)\cap S=\varnothing$.
Hence each such $y$ lies in $L\backslash\left(\bigcup_{s\in S}Bs\right)$,
so that the measure of this set is at least $v$. On the other hand
by Proposition \ref{prop:spec-norm-covering} we have $v\leq\frac{W^{2}}{v}$
or $v\leq W$, which contradicts our assumption.
\end{proof}
\begin{rems*}
\begin{enumerate}
\item In the continuous group case when $W\approx\nicefrac{1}{\sqrt{\left|S\right|}}$
is as small as it can possibly be, Corollary \ref{cor:ball-cover}
yields that every ball of volume $\left|S\right|^{-1/2}$ contains
a point from $S$. While in some cases this covering volume bound
might be far from the truth for $S$, there are very natural cases
(see below) where this is sharp.
\item If $L$ is finite then the only balls are singletons and $L$ itself.
In this case if $B=\left\{ e\right\} $, $\mu\left(B\right)=\frac{1}{\left|L\right|}$
and the upper bound in Proposition \ref{prop:spec-norm-covering}
gives a lower bound for $\bigcup_{s\in S}Bs=S$ (almost covering asserts
that $\nicefrac{\left|S\right|}{\left|L\right|}\rightarrow1$). The
covering Proposition \ref{prop:spec-norm-covering} asserts that $S=L$
if $W<\frac{1}{\left|L\right|}$.
\item There are variants of Proposition \ref{prop:spec-norm-covering} for
symmetric spaces (see \cite{bourgain2012local}) and for Cayley graphs
(\cite{sardari2015diameter,lubetzky2016cutoff}). We formulated the
main variance inequality (\ref{eq:variance-ineq}) in terms of the
spectral norm $W$ alone. In some important cases one can exploit
a more explicit version. If $T_{S,a}$ is normal and hence diagonalizable
with an o.n.b.\ $\phi_{j}$, $j=1,2,\ldots$ of $L_{0}^{2}\left(L\right)$,
then the l.h.s.\ of (\ref{eq:variance-ineq}) is equal to
\begin{equation}
\sum_{j=1}^{\infty}\left|\left\langle I_{B},\phi_{j}\right\rangle \right|^{2}\left|\lambda_{j}\right|^{2}\left(T_{S,a}\right)\label{eq:variance-normal}
\end{equation}
where $\lambda_{j}$ is the eigenvalue corresponding to $\phi_{j}$.
(\ref{eq:variance-ineq}) then follows by invoking $\left|\lambda_{j}\right|\leq W$.
In the form (\ref{eq:variance-normal}) one can allow exceptions to
this as long as there aren't too many such $j$'s and one can also
remove very high frequency $j$'s (in the continuous setting) since
$\left|\left\langle I_{B},\phi_{j}\right\rangle \right|^{2}$ is very
small for these. Such an analysis is carried out in \cite{Sarnak2015Appendixto2015}
in a different setting.
\end{enumerate}
\end{rems*}

We first apply Proposition \ref{prop:spec-norm-covering} to the question
of the covering properties of the arguments of Gaussian primes $P$
in $\mathbb{Z}\left[\sqrt{-1}\right]$, as discussed in Section \ref{sec:Sums-of-squares}.
Consider the set $S_{X}$ of such primes $P$ with $N\left(P\right)\leq X$
(we consider $P,-P,iP,-iP$ as distinct), so that by the Prime Number
Theorem
\[
\left|S_{X}\right|\sim\frac{4X}{\log X}.
\]
Their arguments $\theta_{P}$ are elements in the circle $\left[0,2\pi\right)$
which is the group $L$ for which we apply Proposition \ref{prop:spec-norm-covering}
and its Corollary. For each $P$ we give a weight $a_{P}=\log N\left(P\right)\left(1-\frac{N\left(P\right)}{X}\right)$.
Hence
\begin{align*}
\nu & =\frac{1}{\Psi_{1}\left(X\right)}\sum_{N\left(P\right)\leq X}a_{P}\delta_{\theta\left(P\right)}\\
\Psi_{1}\left(X\right) & =\sum_{N\left(P\right)\leq X}\log N\left(P\right)\left(1-\frac{N\left(P\right)}{X}\right)\sim c_{1}X,\ c_{1}>0.
\end{align*}
The point about these weights is that assuming the generalized Riemann
Hypothesis for the Hecke L-functions, $L\left(s,\lambda^{m}\right)$
where $\lambda^{m}\left(\alpha\right)=\left(\frac{\alpha}{\left|\alpha\right|}\right)^{4m}$
we have (see \cite{Sarnak1985numberpointscertain}) that 
\[
\sum_{N\left(P\right)\leq X}\lambda^{m}\left(\theta_{P}\right)a_{P}\ll\left(\log\left|m\right|+1\right)X^{1/2}.
\]
Hence for the o.n.b.\ $\lambda^{m}$, $m\in\mathbb{Z}$ of $L^{2}\left[0,2\pi\right)$
\[
\nu_{S,a}*\lambda^{m}=\widehat{\nu_{S,a}}\left(m\right)\cdot\lambda^{m}
\]
where
\[
\left|\widehat{\nu_{S,a}}\left(m\right)\right|\ll X^{-1/2}\left(\log\left|m\right|+1\right).
\]
This is an example where the high frequencies can be estimated easily,
and we find that (\ref{eq:variance-normal}) is
\[
\ll X^{-1}\sum_{m\neq0}\left(\log\left|m\right|+1\right)^{2}\left|\left\langle I_{B},\lambda^{m}\right\rangle \right|^{2}\ll X^{-1}\left(\log X\right)^{2}\mu\left(B\right).
\]
Hence for $\frac{\left(\log X\right)^{2}}{X\mu\left(B\right)}=o\left(1\right)$
or $\frac{\mu\left(B\right)\left|S_{X}\right|}{\log X}\rightarrow\infty$;
$B,S_{X}$ is almost covering. That is almost all intervals of length
$h$ with $\frac{h\left|S_{X}\right|}{\log X}\rightarrow\infty$,
have a $\theta_{P}$ with $P\in S_{X}$. So except for the extra $\log X$
factor this is sharp. This has direct bearing to the randomized reduction
in Section \ref{sec:Sums-of-squares}. Namely to find the $P_{j}$'s
there in the requisite sector, we can choose a random angle $\theta$
therein and then according to the above there will be a prime in a
much smaller sector of area $\left(\log X\right)^{2}$. We can check
efficiently each of the few integer lattice points in this sector
to see if they are prime. Of course this is still a randomized reduction.

Since the almost covering length of the $\theta_{P}$'s is optimally
small (under GRH) a natural question is how small is the covering
length? Proposition \ref{prop:spec-norm-covering} coupled with GRH
which gives (essentially) that $W\leq\frac{1}{\left|S_{X}\right|^{1/2}}\left(\log X\right)^{1/2}$,
shows that the covering length is at most $\sqrt{\log X/\left|S_{X}\right|}$.
It turns out that this upper bound is very close to optimal. Let $\alpha=x_{1}+iy_{1}$,
$\beta=x_{2}+iy_{2}$ be two distinct primitive elements of $\mathbb{Z}\left[\sqrt{-1}\right]$
with arguments $\theta_{\alpha},\theta_{\beta}$ in $\left[0,\nicefrac{\pi}{4}\right)$.
Then 
\[
\left|\tan\theta_{\alpha}-\tan\theta_{\beta}\right|=\left|\frac{y_{1}}{x_{1}}-\frac{y_{2}}{x_{2}}\right|\geq\frac{1}{x_{1}x_{2}}\geq\frac{1}{\left|\alpha\right|\left|\beta\right|}.
\]
Hence if $\beta$ is fixed of small norm and $N\left(\alpha\right)\leq X$,
then 
\[
\left|\tan\theta_{\alpha}-\tan\theta_{\beta}\right|\gg\frac{1}{\sqrt{X}}
\]
so there is an interval of length $\nicefrac{1}{\sqrt{X}}$ free of
$\theta_{\alpha}$'s for $\alpha$ primitive and hence certainly free
of $\theta_{P}$'s with $P$ a prime. Thus our upper bounds of $\frac{g\left(X\right)\log X}{\left|S_{X}\right|}$
with $g\rightarrow\infty$, for the almost covering length, and of
$\sqrt{\log X/\left|S_{X}\right|}$ for the covering length, are sharp
up to the $\log$ factors.

\subsection{\label{subsec:Quaternions-and-Ramanujan}Quaternions and Ramanujan}

Let $H\left(\mathbb{R}\right)$ denote the Hamilton quaternions $\alpha=x_{0}+x_{1}\underline{i}+x_{2}\underline{j}+x_{3}\underline{k}$,
$x_{j}\in\mathbb{R}$. The projection $\alpha\mapsto\widetilde{\alpha}=\nicefrac{\alpha}{\left|\alpha\right|}$
is a morphism of $H^{\times}\left(\mathbb{R}\right)$ onto $H^{1}\left(\mathbb{R}\right)$,
the quaternions of norm $1$. Moreover $s$ given by
\begin{equation}
\alpha\mapsto s\left(\alpha\right)=\left[\begin{matrix}x_{0}+ix_{1} & x_{2}+ix_{3}\\
-x_{2}+ix_{3} & x_{0}-ix_{1}
\end{matrix}\right]\label{eq:H1-SU-split}
\end{equation}
is an isomorphism of $H^{1}\left(\mathbb{R}\right)$ with $SU\left(2\right)$,
and both with their bi-invariant metrics are isometric to $S^{3}$.
In this way the solutions to (\ref{eq:4-squares}) which we denote
by $S\left(n\right)$ give points in $G=SU\left(2\right)\simeq S^{3}$
and we can apply the results in Section \ref{subsec:Spectral-gap-and-SA}
to study their covering properties w.r.t.\ balls. We assume that
$n$ is odd. Jacobi showed that (see \cite{davidoff2003elementary})
\[
\left|S\left(n\right)\right|=8\sum_{d\mid n}d.
\]
For the convolution operator on $L^{2}\left(SU\left(2\right)\right)$we
take $a_{s}=\frac{1}{\left|S\left(n\right)\right|}$, $s\in S\left(n\right)$
so that
\[
Tf\left(x\right)=\frac{1}{\left|S\left(n\right)\right|}\sum_{\alpha\in S\left(n\right)}f\left(xs^{-1}\left(\alpha\right)\right).
\]
The Ramanujan Conjectures (Deligne's Theorem \cite{deligne1974conjecture})
imply that (see \cite{lubotzky1987hecke} and also the discussion
below)
\begin{equation}
W=W_{S\left(n\right)}\leq\frac{n^{1/2}\sum_{d\mid n}1}{\left|S\left(n\right)\right|}.\label{eq:weyl-Ramanujan-bound}
\end{equation}
Hence 
\[
\frac{W^{2}}{\left|S\right|\mu\left(B\right)}\leq\frac{\sum_{d\mid n}1}{8\sum_{d\mid n}\frac{1}{d}}\cdot\frac{1}{\left|S\right|\mu\left(B\right)},
\]
and if 
\[
\frac{\left|S\left(n\right)\right|\mu\left(B\right)}{\sum_{d\mid n}1}\rightarrow\infty,
\]
then $S\left(n\right),B$ is almost covering $SU\left(2\right)$.
Two special cases of interest are: $n$ a prime; in which $\left|S\left(n\right)\right|\mu\left(B\right)\rightarrow\infty$
suffices and matches (\ref{eq:random-almost-cover}), while for the
case $n=p^{k}$ for $p>2$ a fixed prime, $\frac{\mu\left(B\right)\left|S\left(n\right)\right|}{\log_{p}\left|S\left(n\right)\right|}\rightarrow\infty$
suffices to almost cover. The last is the case of interest for the
covering properties of Golden Gates and up to the $\log$ factor it
is optimal.

As far as the covering properties of $SU\left(2\right)$ by $S\left(n\right)$,
we apply (\ref{eq:weyl-Ramanujan-bound}) together with Corollary
\ref{cor:ball-cover} to get that
\begin{equation}
\mu\left(B\right)>\left(\frac{\sum_{d\mid n}1}{\sum_{d\mid n}d^{-1/2}}\right)\left|S\left(n\right)\right|^{-1/2}\label{eq:covering-SU2}
\end{equation}
suffices for $S\left(n\right)$ to cover $SU\left(2\right)$ with
balls of volume $\mu\left(B\right)$.

A very interesting question is to determine the covering volume for
$S\left(n\right)$, or at least its exponent. An elementary argument
using repulsion of the projections of $H\left(\mathbb{Z}\right)$
points onto $S^{3}$ (see \cite{sarnak2015letter}) shows that there
are balls $B$ of volume $\left|S\left(n\right)\right|^{-3/4}$ which
are free of any points from $S\left(n\right)$. Hence the exponent
for the covering volume lies in the interval $\left[-\frac{3}{4},-\frac{1}{2}\right]$
and it is probably equal to $-\nicefrac{3}{4}$. One can establish
(\ref{eq:covering-SU2}) with exponent $\left|S\left(n\right)\right|^{-\nicefrac{1}{2}+\varepsilon}$
using the 'Kloosterman circle method', see \cite{Sardari2015Optimalstrongapproximation}.
In this approach Kloosterman sums and estimates associated with them
play a central role and assuming a natural conjecture about cancellations
of sums involving these, one can establish that the volume covering
exponent is $-\nicefrac{3}{4}$ \cite{browning2016twisted}. That
there are such points which are badly approximable by $S\left(n\right)$
leads to a rich metric diophantine approximation theory in this and
much more general contexts, see \cite{Ghosh2014MetricDiophantineapproximation}.

We turn to these questions for integers in a number field. Let $K$
be a totally real number field of degree $k$ and denote by $\sigma_{1},\ldots,\sigma_{k}$
its Galois embeddings into $\mathbb{R}$ with with $\sigma_{1}=\mathrm{identity}$.
For a definite quadratic form $F\left(x_{1},x_{2},x_{3},x_{4}\right)$
such as the sum of four squares we can use Hilbert modular form theory
as above to study the distribution of the $S\left(m\right)$ solutions
to 
\begin{equation}
F\left(x_{1},x_{2},x_{3},x_{4}\right)=x_{1}^{2}+x_{2}^{2}+x_{3}^{2}+x_{4}^{2}=m\quad\text{with}\quad x_{j}\in\mathcal{O}_{K}\quad\text{and}\quad m\in\mathcal{O}_{K}.\label{eq:quadratic-quartic-OK}
\end{equation}
Let $U$ denote the group of units of $\mathcal{O}_{K}$. Now $x=\left(x_{1},x_{2},x_{3},x_{4}\right)\in S\left(m\right)$
iff $\varepsilon x\in S\left(\varepsilon^{2}m\right)$ for any $\varepsilon\in U$.
Thus any properties of $S\left(m\right)$ that are of interest to
us are the same for $S\left(m\right)$ and $S\left(\varepsilon^{2}m\right)$.

Except for very special $K$'s (see \cite{kirschmer2010algorithmic})
the class number of the form $F$ above is not one so that the exact
number of solutions to (\ref{eq:quadratic-quartic-OK}) is not a simple
divisor function of $\left(m\right)$. However the asymptotic behavior
of $\left|S\left(m\right)\right|$ for $\left(m\right)$ (the principal
ideal generated by $m$) which does not have prime factors in the
finitely many ramified primes of $F$ is given by a divisor sum (\cite{schulze2004representation}):
\[
\left|S\left(m\right)\right|\sim C\sum_{a\mid\left(m\right)}N\left(a\right),
\]
here $C$ is a positive constant related to the finite group $\mathrm{Aut}_{\mathcal{O}}\left(F\right)$
and $N\left(a\right)$ is the norm of the ideal $a$.

To each $x\in S\left(m\right)$ and $j=1,\ldots,k$, $\nicefrac{\sigma_{j}\left(x\right)}{\sqrt{\sigma_{j}\left(m\right)}}$
(note $\sigma_{j}\left(m\right)>0$ if $S\left(m\right)\neq\varnothing$)
lies in $S^{3}$, which we have identified with $SU\left(2\right)$.
Hence via the diagonal embedding each $x\in S\left(m\right)$ gives
a point in $L\cong\left(S^{3}\right)^{k}$. The analysis at the beginning
of this Section and the Ramanujan Conjectures which are valid here
as well (\cite{harris2001geometry}) show that 
\[
W_{S\left(m\right)}\leq\frac{N\left(\left(m\right)\right)^{1/2}\sum_{a\mid\left(m\right)}1}{\left|S\left(m\right)\right|}.
\]
Applying Proposition \ref{prop:spec-norm-covering} to $L$ we deduce
that if
\begin{equation}
\frac{\mu\left(B\right)\left|S\left(m\right)\right|}{\sum_{a\mid\left(m\right)}1}\rightarrow\infty\label{eq:Ramanujan-covering}
\end{equation}
then $B,S\left(m\right)$ almost covers $L$.

As before this is essentially an optimal covering. Two two extreme
cases of interest are if $B$ is a ball about $e\in L$ of radius
$r$ so that $\mu\left(B\right)\sim cr^{3k}$, then we require $N\left(m\right)$
to be a little bigger than $r^{-3k}$ in order for almost all balls
of this radius to contain a point from $S\left(m\right)$. The other
case which is of most interest to us is to fixate on the first factor
$\sigma_{1}$. Take $B$ to be $B_{1}\times S^{3}\times\cdots\times S^{3}$
with $B_{1}$ a ball of radius $r$ about $e_{1}$ in $SU\left(2\right)$.
So $\mu\left(B\right)=\mu_{1}\left(B_{1}\right)$ and as long as $\mu\left(B_{1}\right)N\left(m\right)$
goes to infinity as required then $\sigma_{1}\left(x\right),$ $x\in S\left(m\right)$
almost covers $SU\left(2\right)$. So in this case we need $N\left(m\right)$
to be a little larger than $r_{1}^{-3}$.

The above, when applied to special cases, suffices for the analysis
of the optimal covering properties of Golden Gate circuits. For the
purpose of navigation and exact circuit length we need to exploit
special quadratic forms of class number one and which arise as norm
forms of quaternion algebras. We give a different treatment of (\ref{eq:Ramanujan-covering})
from the point of view of the local and global arithmetic of special
quaternion algebras over $K$. This also brings out the important
special finite subgroups $C$ of $PU\left(2\right)$ which are used
to define our Super-Golden Gates.

Let $K$ be as above and $K_{j}$ the real completion of $K$ at $\sigma_{j}$.
For $P$ a prime (ideal) in $\mathcal{O}_{K}$ we let $K_{P}$ be
the completion of $K$ at $P$. Let $D$ be a totally definite quaternion
algebra over $K$, that is $D\otimes K_{j}\cong H\left(\mathbb{R}\right)$
for all $j$, where $H$ is the standard Hamilton quaternion algebra.
The primes at which $D$ is unramified are those for which $D\otimes K_{P}\cong Mat_{2}\left(K_{P}\right)$.
The ramified primes for $D$ are even in number (if one include the
$K_{j}$'s in the count). Choose an order $\mathcal{O}$, which we
usually take to be maximal, in $D$ (for the basic properties and
definitions see \cite{Vigneras1980Arithmetiquedesalgebres}). Let
$D\left(\mathbb{A}\right)$ denote the corresponding adele ring of
$D$ w.r.t.\ the order $\mathcal{O}$. For $P$ unramified and $S=\left\{ P\right\} $
we are interested in the $S$-arithmetic group
\[
\Gamma_{S}=\left\{ \gamma\in D\left(K\right)\,\middle|\,\gamma\text{ is }\mathcal{O}\text{-integral outside }S\right\} .
\]
Now $\nicefrac{D^{\times}\left(K_{P}\right)}{center}\cong PGL_{2}\left(K_{P}\right)$,
and via this identification, $\Gamma_{S}$ projects onto a lattice
$\Gamma$ in $PGL_{2}\left(K_{P}\right)$ (see \cite{Vigneras1980Arithmetiquedesalgebres}).
Also $\nicefrac{H^{\times}\left(\mathbb{R}\right)}{center}\cong PU\left(2\right)$,
so using the embeddings $\sigma_{j}$, $j=1,2,\ldots,k$ we get a
diagonal embedding of $\Gamma=\nicefrac{\Gamma_{S}}{center}$ into
$L\times PGL_{2}\left(K_{P}\right)$, where $L=\left(PU\left(2\right)\right)^{k}$.

The quotient space $X_{P}=\nicefrac{PGL_{2}\left(K_{P}\right)}{PGL_{2}\left(\mathcal{O}_{P}\right)}$
is a $N\left(p\right)+1$ regular tree (\cite{serre1980trees}) on
which $\Gamma$ acts isometrically. Let $L^{2}\left(\Gamma\backslash\left(L\times X_{P}\right)\right)$
be the space of $L^{2}$ (w.r.t.\ $dg\times dg_{P}$) $\Gamma$-periodic
functions:
\[
f\left(\gamma g,\gamma x\right)=f\left(g,x\right)\text{ for }\gamma\in\Gamma,g\in L,x\in X_{P}.
\]
The Hecke operators $T_{t}$, $t>1$ acting on the second variable
are defined by 
\[
T_{t}f\left(g,x\right)=\sum_{d\left(y,x\right)=t}f\left(g,y\right).
\]
$T_{k}$ preserves $L^{2}\left(\Gamma\backslash\left(L\times X_{P}\right)\right)$
and also $L_{0}^{2}\left(\Gamma\backslash\left(L\times X_{P}\right)\right)$;
the orthocomplement of the constant function. It is self adjoint and
its eigenvalues $\lambda$ on $L_{0}^{2}\left(\Gamma\backslash\left(L\times X_{P}\right)\right)$
satisfy 
\begin{equation}
\left|\lambda\right|\leq\left(t+1\right)N\left(P\right)^{t/2}.\label{eq:lambda-Ramanujan-bound}
\end{equation}
One way to derive (\ref{eq:lambda-Ramanujan-bound}) is to use automorphic
representations. The eigenfunctions in $L_{0}^{2}\left(\Gamma\backslash\left(L\times X_{P}\right)\right)$
can, after a further diagonalization of Hecke operators at other primes,
be embedded as a vector in an (non one-dimensional) irreducible representation
$\pi$ of $D\left(\mathbb{A}\right)$ occuring in its right regular
representation on $L^{2}\left(D\left(k\right)\backslash D\left(\mathbb{A}\right)\right)$.
The Ramanujan Conjectures, which are theorems in this context - thanks
to the Jacquet-Langlands correspondence and Deligne's theorem (see
\cite{jacquet1972automorphic,deligne1974conjecture}), assert that
such a $\pi\cong\bigotimes_{v}\pi_{v}$ with $v$ ranging over all
places $v$ of $K$, has each $\pi_{v}$ tempered (\cite{satake1966spherical}).
In particular, this applies to the spherical representation $\pi_{P}$
of $PGL_{2}\left(F_{P}\right)$ and this in turn implies that the
eigenvalue $\lambda$ satisfies (\ref{eq:lambda-Ramanujan-bound}).

To connect this with circuit lengths of elements $\gamma\in\Gamma$
we impose a strong condition, namely that the action of $\Gamma$
on $X_{P}$ is transitive. For this to happen for $P$ outside a finite
set of primes, the class number of the order $\mathcal{O}$ needs
to be $1$. In this definite quaternion algebra setting this happens
only for finitely many $D$'s and $K$'s which have been enumerated
in \cite{kirschmer2010algorithmic}.

In this transitive case $L^{2}\left(\Gamma\backslash\left(G\times X_{P}\right)\right)$
may be identified with functions $h$ on $G$ which are $U_{\Gamma}$
invariant where $U_{\Gamma}=\left\{ \gamma\in\Gamma\,\middle|\,\gamma e=e\right\} $,
$e$ the identity coset $PGL_{2}\left(\mathcal{O}_{P}\right)$ in
$X_{P}$. More explicitly for $f\in L^{2}\left(\Gamma\backslash\left(G\times X_{P}\right)\right)$
\begin{align*}
h\left(g\right) & =g\left(g,e\right),\\
\text{and }h\left(\delta g\right) & =h\left(g\right)\text{ for }\delta\in U_{\Gamma}.
\end{align*}
Such an $h$ has a unique $\Gamma$ extension to $G\times X_{P}$.
The action of $T_{k}$ on $L^{2}\left(U_{\Gamma}\backslash G\right)$
becomes 
\[
T_{k}h\left(g\right)=\sum_{{\gamma\in\Gamma/U_{\Gamma}\atop d\left(\gamma e,e\right)=k}}h\left(g\gamma\right)
\]
and according to (\ref{eq:lambda-Ramanujan-bound}) $T_{k}$ is the
kind of operator considered in (\ref{eq:Hecke-operator}), with a
sharp spectral bound.

Finally in this $X_{P}$ transitive setting we relate the distance
moved by $\gamma$ on the tree to circuit (word) length for generators
of $\Gamma$. Let $\delta_{1},\ldots,\delta_{r}\in\Gamma$ with $r=N\left(P\right)+1$,
be elements that take $v_{0}$ to its $r$ immediate neighbors $S_{1}=\left\{ \xi\,\middle|\,d\left(\xi,v_{0}\right)=1\right\} $
in $X_{P}$. If $y\in X_{P}$ and $y=\beta v_{0}$ with $\beta\in\Gamma$
then $\beta\delta_{j}v_{0}$ for $j=1,\ldots,r$ are the $r$ immediate
neighbors of $y$. Hence if $\mu\in\Gamma$ and $d\left(\mu v_{0},v_{0}\right)=d\left(\mu^{-1}v_{0},v_{0}\right)=k\geq1$,
then there is a unique $j_{1}$ such that $d\left(\mu^{-1}\delta_{j_{1}}v_{0},v_{0}\right)=k-1$.
If $k-1\geq1$ we repeat this to find $\delta_{j_{2}}$ such that
$d\left(\mu^{-1}\delta_{j_{1}}\delta_{j_{2}}v_{0},v_{0}\right)=k-2$.
Repeat this until we arrive at $d\left(\mu^{-1}\delta_{j_{1}}\ldots\delta_{j_{k}}v_{0},v_{0}\right)=0$,
and hence 
\begin{equation}
\mu^{-1}\delta_{j_{1}}\ldots\delta_{j_{k}}=u\in U,\quad\text{or}\quad\mu=\delta_{j_{1}}\ldots\delta_{j_{k}}u^{-1}.\label{eq:path-in-Gamma}
\end{equation}
Since at every step the choices are determined we have that any $\mu\in\Gamma$
with $d\left(\mu v_{0},v_{0}\right)=k$ has a unique expression in
the form (\ref{eq:path-in-Gamma}). 

Thus the distance moved by $\gamma$ on the tree corresponds to the
word length in generators $\delta_{1},\ldots,\delta_{r}$ and the
``Hecke orbit'' corresponds to a circuit of a given length. Moreover
if we are given $\gamma\in\Gamma$ and want to express $\gamma$ in
the form (\ref{eq:path-in-Gamma}) then this can be done efficiently
by navigating the tree, as described above. This is an important feature
that will be exploited in the navigation algorithms below.

In the case that $U_{\Gamma}=\left\{ 1\right\} $, that is when $\Gamma$
acts simply transitively on $X_{P}$, the set $S=\left\{ \delta_{1},\ldots,\delta_{r}\right\} $
is invariant under $s\mapsto s^{-1}$. Hence $S=\left\{ \delta_{1},\ldots,\delta_{t},\mu_{1},\mu_{1}^{-1},\ldots,\mu_{v},\mu_{v}^{-1}\right\} $
with $\delta_{j}=\delta_{j}^{-1}$, $\mu_{j}\neq\mu_{j}^{-1}$ and
$t+2v=r$. (\ref{eq:path-in-Gamma}) then shows that $X_{P}$ is the
Cayley graph of $\Gamma$ with respect to the generators $S$, the
$\mu_{j}$'s being of infinite order and $\Gamma\cong\left\langle \mu_{1}\right\rangle *\ldots*\left\langle \mu_{v}\right\rangle *\left\langle \delta_{1}\right\rangle *\ldots*\left\langle \delta_{t}\right\rangle $.
The distance on the tree $X_{P}$ corresponding to the word (circuit)
length w.r.t.\ the symmetric set of generators $S$. This is the
case of Golden-Gate sets which as we have noted can arise from only
a finite number of $D$'s and $K$'s; see Section \ref{sec:Super-Golden-Gates}
for some examples (by varying $P$ there are infinitely many such
Golden Gate sets).

For our super Golden Gate sets we require further that $\Gamma$ acts
transitively on the edges of $X_{P}$. In fact we assume that $U_{\Gamma}$
acts simply transitively on $S_{1}$ and that among the $\delta_{j}$'s
there is an involution $T$. In this case the set $\delta_{1},\ldots,\delta_{r}$
can be taken as $u^{-1}Tu$, $u\in U_{\Gamma}$ (note $\left|U_{\Gamma}\right|=r$).
(\ref{eq:path-in-Gamma}) now asserts that $\Gamma\cong U_{\Gamma}*\left\langle T\right\rangle $
and $d\left(\gamma v_{0},v_{0}\right)$ is exactly the $T$-count
in the representation of $\gamma\in\Gamma$ as a product of $T$ and
$u$'s for $u\in U$. This gives rise to our super-golden gate sets
$C$ plus $T$, with $C=U_{\Gamma}$. There are only finitely many
of these since the set of $D$'s and $K$'s lie in a finite list as
does the set of orders $\mathcal{O}$. Hence $U_{\Gamma}$ which is
the unit group of the order is limited, and this limits the primes
$P$ for which this group can act transitively on $S_{1}$. We conclude
that there are only finitely many possibilities for super Golden Gate
sets. There are a number of very interesting examples where all of
this happens, and these are described in the next section.

We end this Section by describing the relation between word length
of elements of $\Gamma$ in terms of the generators $\mu_{1},\ldots,\mu_{v},\delta_{1},\ldots,\delta_{t}$
and strong approximation for (\ref{eq:4-squares}). We explain this
for the case that $D=H=\left\langle 1,\underline{i},\underline{j},\underline{k}\right\rangle $
the Hamilton quaternions over $\mathbb{Q}$, this being the case that
was highlighted in \cite{lubotzky1986hecke} and is the basic example
of what we are calling Golden-Gates. Let $H\left(\mathbb{Z}\right)$
be the order $\mathbb{Z}+\mathbb{Z}\underline{i}+\mathbb{Z}\underline{j}+\mathbb{Z}\underline{k}$
in $H$. The group $U$ of invertible elements in $H\left(\mathbb{Z}\right)$
is $\left\{ \pm1,\pm\underline{i},\pm\underline{j},\pm\underline{k}\right\} $.
The spheres $S\left(m\right)=\left\{ \alpha\in H\left(\mathbb{Z}\right):\nr\left(\alpha\right):=\alpha\overline{\alpha}=m\right\} $
correspond to the solutions to (\ref{eq:4-squares}). These spheres
are acted on by $U\times U$ by left and right multiplication
\begin{equation}
\alpha\mapsto u_{1}\alpha u_{2}.\label{eq:UxU-action}
\end{equation}
For $p$ and odd prime $\left|S\left(p\right)\right|=8\left(p+1\right)$
and after multiplying on the right by $u\in U$ there are $\left(p+1\right)$
solutions. If $p\equiv1\,\left(4\right)$ exactly one coordinate of
an $\alpha\in S\left(p\right)$ is odd and we use $U$ to make it
the first coordinate and also ensure that it is positive. This gives
us $\mu_{1},\overline{\mu_{1}},\mu_{2},\overline{\mu_{2}},\ldots,\mu_{v},\overline{\mu_{v}}$
with $2v=p+1$ as the solutions (see \cite{lubotzky1987hecke,davidoff2003elementary}).
If $p\equiv3\,\left(4\right)$ then exactly one of the coordinates
of an $\alpha\in S\left(p\right)$ is even and we can use $U$ to
arrange that this coordinate is the first one, and again we can act
by $\pm1$ to choose the sign. In this way the $\left(p+1\right)$
solutions obtained become $\mu_{1},\overline{\mu_{1}},\ldots,\mu_{v},\overline{\mu_{v}},\delta_{1},\ldots,\delta_{t}$
with $p+1=2v+t$ and $\delta_{j}^{2}=-p$ for $j=1,\ldots,t$ (see
\cite{davidoff2003elementary}).

Now $H$ is unramified at each odd prime $p$ so that the group $\Gamma$
generated by these $p+1$ elements can be realized as a subgroup of
$PGL_{2}\left(\mathbb{Q}_{p}\right)$. These generators take the identity
coset $v_{0}=PGL_{2}\left(\mathbb{Z}_{p}\right)\in X_{p}=\nicefrac{PGL_{2}\left(\mathbb{Q}_{p}\right)}{PGL_{2}\left(\mathbb{Z}_{p}\right)}$
to its $\left(p+1\right)$ neighbors, and $\Gamma$ acts simply transitively
on $X_{p}$. Moreover $\Gamma\cong\left\langle \mu_{1}\right\rangle *\ldots\left\langle \mu_{v}\right\rangle *\left\langle \delta_{1}\right\rangle *\ldots*\left\langle \delta_{t}\right\rangle $
with $\left\langle \mu_{j}\right\rangle \cong\mathbb{Z}$ and $\left\langle \delta_{j}\right\rangle \cong\nicefrac{\mathbb{Z}}{2\mathbb{Z}}$,
and $X_{p}$ is the Cayley graph of $\Gamma$ w.r.t.\ $\mu_{1},\overline{\mu_{1}},\ldots,\mu_{v},\overline{\mu_{v}},\delta_{1},\ldots,\delta_{t}$.
Now a reduced word of length $\ell$ in these generators yields a
unique (up to right multiplication by $U$) $\alpha$ in $S\left(p^{\ell}\right)$,
which is primitive, namely, $\gcd\left(x_{1},\ldots,x_{4}\right)=1$.
Furthermore, every primitive $\alpha$ in $S\left(p^{\ell}\right)$
is achieved this way. This identifies the distance $d_{X_{p}}\left(v_{0},\alpha v_{0}\right)$
with the word length of $\alpha$ in our generators. The splitting
of $\nicefrac{H\left(\mathbb{R}\right)}{\mathbb{R}_{+}}$ by the isomorphism
(\ref{eq:H1-SU-split}) to $G=PU\left(2\right)$, yields the subgroup,
also denoted by $\Gamma$, generated by our $p+1$ elements. The question
of how well the words of length $\ell$, $\ell\leq L$ cover $G$
is equivalent to the same question for the $\alpha\in S\left(p^{\ell}\right)$
covering $S^{3}$, $\ell\leq L$. That is the problem discussed at
the beginning of this Section. It follows from these that the circuits
of length $\ell\leq L$ almost cover $G$ optimally. Precisely, the
number of circuits of length $\ell\leq L$ in $G$ is $\approx p^{L}$
(here $p$ is \emph{fixed} and $L\rightarrow\infty$) and as long
as 
\[
\frac{V\,p^{L}}{L}\rightarrow\infty,
\]
then almost all balls $B_{V}\left(\xi\right)$ in $G$ are covered.

Finally as far as navigating $G$ with these generators, again the
problem reduces to one of strong approximation, namely Task \ref{task:ball}
of Section \ref{sec:Sums-of-squares}, with $n=p^{\ell}$. The task
needs to be resolved in $\mathrm{poly}\left(L\right)$ steps. So to
determine if there is a circuit of length $\ell\leq L$ in $B_{V}\left(\xi\right)$,
we can do so for each $\ell\leq L$ separately. The question for a
given such $\ell$ decouples into two steps, firstly finding an $\alpha\in S\left(p^{\ell}\right)$
which is in $B_{V}\left(\xi\right)$. According to Theorem \ref{thm:ross-selinger}
this can be done efficiently at least if $\xi$ is special in $S^{3}$,
which amounts to $\xi\in G$ being diagonal. Once we have found $\alpha$
(or determined that none exists) we can express $\alpha$ as a circuit
of length $\ell$ in the generators by viewing $\alpha\in\nicefrac{PGL_{2}\left(\mathbb{Q}_{p}\right)}{PGL_{2}\left(\mathbb{Z}_{p}\right)}$.
There is a unique generator which will move $\alpha$ one step closer
to $v_{0}$. Repeating this $\ell$ times gives an efficient factorization
of $\alpha$. 

To summarize, given $\xi\in G$ diagonal and $B_{V}\left(\xi\right)$,
then assuming that we can factor efficiently and the heuristic in
the discussion in Theorem \ref{thm:ross-selinger}, we can find efficiently
a circuit of least word length in our generators which lies in $B_{V}\left(\xi\right)$.

Without an efficient factorization algorithm we settle for finding
special solutions to the strong approximation problem as described
in remark \ref{rem:without-factorization} and otherwise proceed as
above. This secures a circuit which lies in whose word length is;
\begin{equation}
\left(1+o\left(1\right)\right)\text{ times longer than the optimally short circuit}.\label{eq:one-plus-o-one}
\end{equation}

The above allows us to navigate to diagonal $\xi$'s in $G$ optimally
and efficiently. As far as an optimal navigation to a general $\xi$
via strong approximation, the problem at least in the form of Task
\ref{task:main-task} is apparently hard. We can navigate efficiently
to a $\xi=\left(\xi_{1},\xi_{2},\xi_{3},\xi_{4}\right)\in S^{3}$
which has at least two coordinate equal to $0$. In particular to
elements $x\in G$ of the form $\mu_{1}\xi\mu_{2}$ with $\mu_{1},\mu_{2}\in C_{4}$
th Pauli group and $\xi$ diagonal. Now any element of $G$ can be
expressed as a product of three such $\mu_{1}\xi\mu_{2}$'s and efficiently,
hence we can navigate efficiently to $x$ with a circuit whose lengthy
is $3$-times longer then the shortest circuit is for the generic
$x$. The efficient navigation with Super Golden Gates and using $T$-counts
is carried out similarly. This completes the theoretical diophantine
and algorithmic analysis concerning Golden and Super-Golden gates
that was announced in the Introduction. We end this Section with two
explicit examples.
\begin{description}
\item [{$\left(A\right)$\ V-gates:}] We take $p=5$. Then $p+1=6$ and
the generators are $1\pm2\underline{i},1\pm2\underline{j},1\pm2\underline{k}$;
which yield $S_{1}=\left(\begin{smallmatrix}1+2i & 0\\
0 & 1-2i
\end{smallmatrix}\right)$, $S_{2}=\left(\begin{smallmatrix}1 & 2\\
-2 & 1
\end{smallmatrix}\right)$, $S_{3}=\left(\begin{smallmatrix}1 & 2i\\
2i & 1
\end{smallmatrix}\right)$ ($\Gamma=\left\langle S_{1},S_{2},S_{3}\right\rangle \cong\mathbb{Z}*\mathbb{Z}*\mathbb{Z}$)
as a Golden Gate set $S_{1}^{\pm1},S_{2}^{\pm1},S_{3}^{\pm1}$ called
$V$-gates.
\item [{$\left(B\right)$}] $p=3$, $v=0$, $t=4$; $\delta_{1}=\underline{i}+\underline{j}+\underline{k}$,
$\delta_{2}=\underline{i}-\underline{j}+\underline{k}$, $\delta_{3}=\underline{i}+\underline{j}-\underline{k}$,
$\delta_{4}=\underline{i}-\underline{j}-\underline{k}$. These yield
the four involutions $S_{j}$ in $G$ given by $S_{1}=\left(\begin{smallmatrix}i & 1+i\\
-1+i & -i
\end{smallmatrix}\right)$, $S_{2}=\left(\begin{smallmatrix}i & -1+i\\
1+i & -i
\end{smallmatrix}\right)$, $S_{3}=\left(\begin{smallmatrix}i & 1-i\\
-1-i & -i
\end{smallmatrix}\right)$, $S_{4}=\left(\begin{smallmatrix}i & -1-i\\
1-i & -i
\end{smallmatrix}\right)$ ($\Gamma\cong\left\langle S_{1}\right\rangle *\left\langle S_{2}\right\rangle *\left\langle S_{3}\right\rangle *\left\langle S_{4}\right\rangle $)
and are a Golden Gate set.
\end{description}
The action of $U\times U$ on $S\left(3\right)$ given by (\ref{eq:UxU-action})
is transitive and this leads to our first example of a super gate
to which we now turn.

\section{\label{sec:Super-Golden-Gates}Super Golden Gates}

Each Super-Golden-Gate set is composed of a finite group $C$ and
an involution $T$, which lie in a $\left\{ P\right\} $-arithmetic
group for $P$ a prime ideal of the integers $\mathcal{O}_{K}$ of
a totally definite quaternion algebra $D$, over a totally real number
field $K$. We require that: 
\begin{enumerate}
\item $C$ acts simply transitively on the neighbors of the origin in $X_{P}$.
\item $T$ is an involution which takes the origin to one of the neighbors
. 
\end{enumerate}
Since $T$ is an involution, it inverts some edge $e_{0}$ whose origin
is $v_{0}$. Denote by $\Gamma$ the group generated by $C$ and $T$,
and by $\Delta$ the group generated by $\left\{ cTc^{-1}\,\middle|\,c\in C\right\} $.
The next Proposition follows by Bass-Serre theory from the assumptions
(1), (2) above.
\begin{prop}
\begin{enumerate}
\item $\Gamma$ acts simply transitively on the directed edges of $X_{P}$.
\item The $T$-count of $\gamma\in\Gamma$ is $d_{X_{p}}\left(v_{0},\mathrm{orig}\left(\gamma e_{0}\right)\right)=d_{X_{p}}\left(v_{0},\gamma v_{0}\right)$,
and in particular there are $\left|C\right|^{2}\left(\left|C\right|-1\right)^{t-1}$
elements of $T$-count $t$.
\item $\Gamma$ is the free product of $C$ and $\left\langle T\right\rangle \simeq\nicefrac{\mathbb{Z}}{2\mathbb{Z}}$.
\item $\Delta$ is a Golden-Gate set; namely, it acts simply transitively
on the vertices of $X_{P}$, and thus is a free product of $\left|C\right|$
copies of $\nicefrac{\mathbb{Z}}{2\mathbb{Z}}$.
\item $\Gamma=\Delta\rtimes C$. 
\end{enumerate}
\end{prop}
One can navigate in $\Gamma$ as follows: If the origin of $\gamma e_{0}$
is $v_{0}$, then $\gamma\in C$. In not, there is a unique $c_{0}\in C$
for which $\gamma c_{0}^{-1}e_{0}$ is pointing towards $v_{0}$,
since $\left\{ \gamma ce_{0}\,\middle|\,c\in C\right\} $ are the
rotations of $\gamma e_{0}$ around its origin. Then, 
\begin{align*}
d_{X_{p}}\left(v_{0},\mathrm{orig}\left(\gamma c_{0}^{-1}Te_{0}\right)\right) & =d_{X_{p}}\left(v_{0},\mathrm{term}\left(\gamma c_{0}^{-1}e_{0}\right)\right)\\
 & =d_{X_{p}}\left(v_{0},\mathrm{orig}\left(\gamma c_{0}^{-1}e_{0}\right)\right)-1=d_{X_{p}}\left(v_{0},\mathrm{orig}\left(\gamma e_{0}\right)\right)-1.
\end{align*}
Continuing in this manner one finds $c_{0},\ldots,c_{\ell-1}\in C$
such that $\mathrm{orig}\left(\gamma c_{0}^{-1}T\ldots c_{\ell-1}^{-1}Te_{0}\right)=v_{0}$,
hence $c_{\ell}:=\gamma c_{0}^{-1}T\ldots c_{\ell-1}^{-1}T$ is in
$C$, and $\gamma=c_{\ell}Tc_{\ell-1}T\ldots Tc_{1}Tc_{0}$.

\subsection{\label{subsec:Examples}Examples}

In the examples of Super-Golden-Gate sets that follow we explicate
the definite quaternion algebra $D$ over the totally real field $K$
and order $\mathcal{O}$ as well as the prime $P$ at which we allow
denominators. Both $K$ and $\mathcal{O}$ have class number $1$
and we refer to the corresponding entries in Table 8.2 of \cite{kirschmer2010algorithmic}.

\subsubsection{\label{subsec:Lipschitz-quaternions}Lipschitz quaternions and Pauli
matrices}

The simplest example occurs for $D$ being the Hamilton quaternions
over $\mathbb{Q}$, and $\mathcal{O}$ the Lipschitz order $\mathbb{Z}+\mathbb{Z}\underline{i}+\mathbb{Z}\underline{j}+\mathbb{Z}\underline{k}$
which was used to construct the golden gates in the previous Section.
If one takes $P=\left(3\right)$, then an explicit isomorphism from
$\left(D\otimes\mathbb{Q}_{3}\right)^{\times}/\mathbb{Q}_{3}^{\times}$
to $PGL_{2}\left(\mathbb{Q}_{3}\right)$ is given by 
\begin{equation}
a+bi+cj+dk\longmapsto\left(\begin{matrix}a+c+\rho d & b+\rho c-d\\
-b+\rho c-d & a-c-\rho d
\end{matrix}\right),\qquad\left(\rho=\sqrt{-2}\right),\label{eq:split-Q3}
\end{equation}
where $\sqrt{-2}=\ldots200211\in\mathbb{Q}_{3}$. The unit group $\mathcal{O}^{\times}/\mathbb{Z}^{\times}$
is $\left\{ 1,i,j,k\right\} $, which corresponds in $PU\left(2\right)$
to the Pauli matrices
\[
C_{4}:=\left\{ \left(\begin{matrix}1 & 0\\
0 & 1
\end{matrix}\right),\left(\begin{matrix}i & 0\\
0 & -i
\end{matrix}\right),\left(\begin{matrix}0 & 1\\
-1 & 0
\end{matrix}\right),\left(\begin{matrix}0 & i\\
i & 0
\end{matrix}\right)\right\} ,
\]
and it is mapped under (\ref{eq:split-Q3}) to $\left\{ \left(\begin{smallmatrix}1 & 0\\
0 & 1
\end{smallmatrix}\right)\right.$, $\left(\begin{smallmatrix}\phantom{\ldots22}0 & 1\\
\ldots222 & 0
\end{smallmatrix}\right)$, $\left(\begin{smallmatrix}\phantom{\ldots}\phantom{2}\phantom{2}1 & \ldots211\\
\ldots211 & \ldots222
\end{smallmatrix}\right)$, $\left.\left(\begin{smallmatrix}\ldots211 & \ldots222\\
\ldots222 & \phantom{2}\ldots12
\end{smallmatrix}\right)\right\} $. These indeed act simply transitively on the neighbors of $v_{0}\in X_{3}=\nicefrac{PGL_{2}\left(\mathbb{Q}_{3}\right)}{PGL_{2}\left(\mathbb{Z}_{3}\right)}$,
as one can verify by multiplying them on the right with $\left(\begin{smallmatrix}1 & 0\\
0 & 3
\end{smallmatrix}\right)$ and computing the Iwasawa decomposition of the result. The element
$i+j+k\in\mathcal{O}$ is an involution in $D^{\times}/\mathbb{Q}^{\times}$,
it inverts the edge $e_{0}=\left[v_{0}\rightarrow\left(\begin{smallmatrix}3 & 1\\
 & 1
\end{smallmatrix}\right)v_{0}\right]$ (again using (\ref{eq:split-Q3})), and it corresponds in $PU\left(2\right)$
to 
\[
T_{4}:=\left(\begin{matrix}1 & 1-i\\
1+i & -1
\end{matrix}\right).
\]

\subsubsection{Hurwitz quaternions and tetrahedral gates}

Keeping $D=\smash{\left(\frac{-1,-1}{\mathbb{Q}}\right)}$, we consider
the maximal order of Hurwitz quaternions 
\[
\mathcal{O}=\mathbb{Z}\left[i,\tfrac{1+i+j+k}{2}\right]=\mathbb{Z}\oplus\mathbb{Z}\cdot i\oplus\mathbb{Z}\cdot j\oplus\mathbb{Z}\cdot\tfrac{1+i+j+k}{2},
\]
which corresponds to the entry $1\ 1\ 2\ 1$ in \cite[Tab.\ 8.2]{kirschmer2010algorithmic}.
The unit group $\mathcal{O}^{\times}/\mathbb{Z}^{\times}$ is the
Platonic tetrahedral group, which is isomorphic to $\mathrm{Alt}_{4}$.
In $PU\left(2\right)$ it corresponds to 
\[
C_{12}:=\left\langle \left(\begin{matrix}i & 0\\
0 & -i
\end{matrix}\right),\left(\begin{matrix}1 & 1\\
i & -i
\end{matrix}\right)\right\rangle .
\]
For $P=\left(11\right)$, the same splitting as in (\ref{eq:split-Q3})
can be used (with $\rho=\sqrt{-2}\in\mathbb{Q}_{11}$), and $C_{12}$
acts simply transitively on the neighbors of the origin of $X_{P}$.
As an involution one can take $3i+j+k$, which maps to 
\[
T_{12}:=\left(\begin{matrix}3 & 1-i\\
1+i & -3
\end{matrix}\right)
\]
in $PU\left(2\right)$. Here and in the examples which follow we do
not write down the explicit action of $C_{k}$ and $T_{k}$ on the
tree $X_{P}$ (where $k=N\left(P\right)+1$). In all the examples
the splitting of $D\otimes K_{P}$ is chosen so that $C_{k}$ fixes
the vertex $PGL_{2}\left(\mathcal{O}_{K_{P}}\right)\in X_{P}$, and
$T_{k}$ invert an edge which is incident to it.

\subsubsection{Octahedral gates}

We now take $D$ to be the Hamilton quaternions over $K=\mathbb{Q}\left(\sqrt{2}\right)$,
which has $\mathcal{O}_{K}=\mathbb{Z}\left[\sqrt{2}\right]$ and $U_{K}=\mathcal{O}_{K}^{\times}=\left\{ \pm\left(1+\sqrt{2}\right)^{m}\,\middle|\,m\in\mathbb{Z}\right\} $.
A maximal order in $D$ is given by 
\[
\mathcal{O}=\mathcal{O}_{K}\oplus\mathcal{O}_{K}\cdot\frac{1+i}{\sqrt{2}}\oplus\mathcal{O}_{K}\cdot\frac{1+j}{\sqrt{2}}\oplus\mathcal{O}_{K}\cdot\frac{1+i+j+k}{2},
\]
which is entry $2\ 8\ 1\ 1$ in \cite[Tab.\ 8.2]{kirschmer2010algorithmic}
(see also \cite[§5]{Vigneras1980Arithmetiquedesalgebres}). The unit
group of $\mathcal{O}$ modulo scalars is the Platonic octahedral
group 
\[
\nicefrac{U}{U_{K}=}\nicefrac{\mathcal{O}^{\times}}{\mathcal{O}_{F}^{\times}}=\left\langle \frac{1+i}{\sqrt{2}},\frac{1+i+j+k}{2}\right\rangle \cong\mathrm{Sym}_{4},
\]
which corresponds in $PU\left(2\right)$ to the Clifford group
\[
C_{24}:=\left\langle \left(\begin{matrix}1 & 0\\
0 & i
\end{matrix}\right),\left(\begin{matrix}1 & 1\\
i & -i
\end{matrix}\right)\right\rangle .
\]
For $P=\left(5-\sqrt{2}\right)$, there is a (unique) continuous isomorphism
$K_{P}=\mathbb{Q}\left(\sqrt{2}\right)_{5-\sqrt{2}}\cong\mathbb{Q}_{23}$,
which is given by sending $\sqrt{2}$ to a square root of $2$ in
$\mathbb{Q}_{23}$, chosen so that $P$ maps to a uniformizer. Then,
a splitting $\left(D\otimes K_{P}\right)^{\times}/K_{P}^{\times}\overset{\sim}{\longrightarrow}PGL_{2}\left(\mathbb{Q}_{23}\right)$
is given by
\[
a+bi+cj+dk\longmapsto\left(\begin{matrix}a+2c+\rho d & b+\rho c-2d\\
-b+\rho c-2d & a-2c-\rho d
\end{matrix}\right),\qquad\left(\rho=\sqrt{-5}\in\mathbb{Q}_{23}\right).
\]
Adding the involution $\left(1+\tfrac{1}{\sqrt{2}}\right)i+\frac{j}{\sqrt{2}}+\left(1-\sqrt{2}\right)k$,
which corresponds to 
\[
T_{24}:=\left(\begin{matrix}-1-\sqrt{2} & 2-\sqrt{2}+i\\
2-\sqrt{2}-i & 1+\sqrt{2}
\end{matrix}\right)\in PU\left(2\right),
\]
gives the super-golden-gate set $\Gamma=\left\langle C_{24},T_{24}\right\rangle $,
which is the full $\left\{ P\right\} $-arithmetic group in $\mathcal{O}$. 

Some subgroups of $\mathrm{Sym}_{4}$ give other super-gate-sets:

\paragraph{\uline{8-gates.}}

For $P=\left(3+\sqrt{2}\right)$, we have $\left(D\otimes K_{P}\right)^{\times}/K_{P}^{\times}\overset{\sim}{\longrightarrow}PGL_{2}\left(\mathbb{Q}_{7}\right)$
by
\[
a+bi+cj+dk\longmapsto\left(\begin{array}{rr}
a+\frac{c\left(1+\rho\right)+d\left(1-\rho\right)}{2} & b+\frac{c\left(1-\rho\right)-d\left(1+\rho\right)}{2}\\
-b+\frac{c\left(1-\rho\right)-d\left(1+\rho\right)}{2} & a-\frac{c\left(1+\rho\right)+d\left(1-\rho\right)}{2}
\end{array}\right),\qquad\left(\rho=\sqrt{-3}\right),
\]
and we obtain the gates
\[
C_{8}:=\left\langle \left(\begin{matrix}0 & 1\\
1 & 0
\end{matrix}\right),\left(\begin{matrix}1 & 0\\
0 & i
\end{matrix}\right)\right\rangle ,\quad T_{8}:=\left(\begin{matrix}\sqrt{2}-1 & 1-\sqrt{2}i\\
1+\sqrt{2}i & 1-\sqrt{2}
\end{matrix}\right)
\]
which correspond to $\mathrm{Dih}_{4}\leq\mathrm{Sym}_{4}$ and the
involution $i+j+\frac{k-i}{\sqrt{2}}$.

\paragraph{\uline{3-gates (ramified $P$).}}

For the prime $P=\left(\sqrt{2}\right)$, one has $K_{P}\cong\mathbb{Q}_{2}\left(\sqrt{2}\right)$,
and a splitting $\left(D\otimes K_{P}\right)^{\times}/K_{P}^{\times}\overset{\sim}{\longrightarrow}PGL_{2}\left(\mathbb{Q}_{2}\left(\sqrt{2}\right)\right)$
is given by
\[
a+bi+cj+dk\longmapsto\left(\begin{matrix}a+b+c\left(\alpha-\beta\right)+d\left(\alpha+\beta\right) & \sqrt{2}\left(b+c\alpha+d\beta\right)\\
-\sqrt{2}\left(b-c\beta+d\alpha\right) & a-b-c\left(\alpha-\beta\right)-d\left(\alpha+\beta\right)
\end{matrix}\right),
\]
where $\alpha=\frac{2+\sqrt{-14}}{3}$ and $\beta=\frac{2\sqrt{2}-\sqrt{-7}}{3}$.
The gate set obtained is
\[
C_{3}:=\left\langle \left(\begin{matrix}1 & 1\\
i & -i
\end{matrix}\right)\right\rangle ,\quad T_{3}:=\left(\begin{matrix}0 & \sqrt{2}\\
1+i & 0
\end{matrix}\right),
\]
which correspond to $\left\langle \frac{1+i+j+k}{2}\right\rangle \leq\mathrm{Sym}_{4}$,
and $\left(1-\frac{1}{\sqrt{2}}\right)j+\frac{1}{\sqrt{2}}k$.

\paragraph{\uline{V-gates (hybrid example).}}

Taking the subgroup $\left\langle \frac{1+i+j+k}{2},\frac{i-j}{\sqrt{2}}\right\rangle \cong\mathrm{Sym}_{3}$
of the octahedral group, one can scale its generators to $1+i+j+k$
and $i-j$, which lie in the Lipschitz quaternions. There they generate
an infinite group, which projects onto a copy of $\mathrm{Sym}_{3}$
in $\left(\frac{-1,-1}{\mathbb{Q}}\right)^{\times}/\mathbb{Q}^{\times}$.
Together with the involution $j+2k$ this gives the gate set
\[
C_{6}:=\left\langle \left(\begin{matrix}1 & 1\\
i & -i
\end{matrix}\right),\left(\begin{matrix}0 & i\\
1 & 0
\end{matrix}\right)\right\rangle ,\quad T_{6}:=\left(\begin{matrix}0 & 2-i\\
2+i & 0
\end{matrix}\right).
\]

\subsubsection{Icosahedral gates}

We move to the golden field $K=\mathbb{Q}\left(\sqrt{5}\right)$,
for which $\mathcal{O}_{K}=\mathbb{Z}\left[\varphi\right]$ ($\varphi=\frac{1+\sqrt{5}}{2}$),
and $U_{K}=\mathcal{O}_{K}^{\times}=\left\{ \pm\varphi^{m}\,\middle|\,m\in\mathbb{Z}\right\} $.
A maximal order in $D:=\left(\frac{-1,-1}{K}\right)$ is given by
the ring of \emph{icosians }
\[
\mathcal{O}=\left\{ \frac{1}{2}\left({\left(a+b\sqrt{5}\right)+\left(c+d\sqrt{5}\right)i\atop +\left(e+f\sqrt{5}\right)j+\left(g+h\sqrt{5}\right)k}\right)\,\middle|\,\begin{matrix}\small{\text{\ensuremath{a+c+e+g\equiv b+d+f+h\equiv0\Mod{2}}}}\\
\small{\text{\ensuremath{\left(c,e,a\right)\equiv\begin{cases}
\left(b,d,f\right)\Mod{2},\quad\text{ or}\\
\left(b,d,f\right)+\left(1,1,1\right)\Mod{2}
\end{cases}}}}
\end{matrix}\right\} ,
\]
which is $2\ 5\ 1\ 1$ in \cite[Tab.\ 8.2]{kirschmer2010algorithmic},
and which forms, with respect to a non-standard quadratic form, a
copy of the $E_{8}$ lattice (c.f.\ \cite{tits1980quaternions} or
\cite[§8]{Conway1999}). The unit group of $\mathcal{O}$ modulo scalars
is the Platonic icosahedral group 
\[
\nicefrac{U}{U_{K}=}\nicefrac{\mathcal{O}^{\times}}{\mathcal{O}_{F}^{\times}}=\left\langle \frac{1+i+j+k}{2},\frac{i+\varphi^{-1}j+\varphi k}{2}\right\rangle \cong\mathrm{Alt}_{5},
\]
which corresponds to 
\[
C_{60}:=\left\langle \left(\begin{matrix}1 & 1\\
i & -i
\end{matrix}\right),\left(\begin{matrix}1 & \varphi-i/\varphi\\
\varphi+i/\varphi & -1
\end{matrix}\right)\right\rangle .
\]
For $P=\left(7+5\varphi\right)$ one has $K_{P}\cong\mathbb{Q}_{59}$,
and the splitting (\ref{eq:split-Q3}) can be used again. As an involution
one can take $\left(2+\varphi\right)i+j+k$, which gives
\[
T_{60}:=\left(\begin{matrix}2+\varphi & 1-i\\
1+i & -2-\varphi
\end{matrix}\right),
\]
and the generated group $\Gamma=\left\langle C_{60},T_{60}\right\rangle $
is the full $\left\{ 7+5\varphi\right\} $-arithmetic group of $\mathcal{O}$.

\paragraph{\uline{Icosahedral $12$-gates.}}

Taking $P=\left(4-\varphi\right)$, one has $\left(D\otimes K_{P}\right)^{\times}/K_{P}^{\times}\overset{\sim}{\longrightarrow}PGL_{2}\left(\mathbb{Q}_{11}\right)$
using (\ref{eq:split-Q3}). This gives another set of gates acting
on a $12$-regular tree:
\[
C_{12}':=\left\langle \left(\begin{matrix}1 & 1\\
i & -i
\end{matrix}\right),\left(\begin{matrix}1 & \varphi+\frac{i}{\varphi}\\
\varphi-\frac{i}{\varphi} & -1
\end{matrix}\right)\right\rangle ,\quad T_{12}':=\left(\begin{matrix}\varphi-1 & 1-i\\
1+i & 1-\varphi
\end{matrix}\right),
\]
corresponding to $\left\langle \frac{1+i+j+k}{2},\frac{i-\varphi^{-1}j+\varphi k}{2}\right\rangle \cong\mathrm{Alt}_{4}\leq\mathrm{Alt}_{5}$
and the involution $\left(\varphi-1\right)i+j+k$.

\paragraph{\uline{5-gates (inert $P$).}}

Taking $P=\left(2\right)$, the completion $K_{P}$ has a residue
field of size four, and a splitting $\left(D\otimes K_{P}\right)^{\times}/K_{P}^{\times}\overset{\sim}{\longrightarrow}PGL_{2}\left(K_{P}\right)$
is given by 
\[
a+bi+cj+dk\longmapsto\left(\begin{matrix}a+c\alpha+d\beta & b+c\beta-d\alpha\\
-b+c\beta-d\alpha & a-c\alpha-d\beta
\end{matrix}\right),\quad\left(\begin{matrix}\alpha=\frac{\sqrt{5}-\sqrt{-7}}{2}\\
\beta=\frac{\sqrt{5}+\sqrt{-7}}{2}
\end{matrix}\right).
\]
The cyclic group $\left\langle \frac{\varphi+i/\varphi+j}{2}\right\rangle \leq\mathrm{Alt}_{5}$
and the involution $j+k$ give the gate set 
\[
C_{5}:=\left\langle \left(\begin{matrix}1+\varphi+i & \varphi\\
-\varphi & 1+\varphi-i
\end{matrix}\right)\right\rangle ,\quad T_{5}:=\left(\begin{matrix}0 & 1\\
i & 0
\end{matrix}\right).
\]

\paragraph{\uline{Nonexamples.}}

For the ramified and inert primes $P=\left(\sqrt{5}\right)$ and $P=\left(3\right)$
of $\mathbb{Z}\left[\varphi\right]$, the corresponding $\left\{ P\right\} $-arithmetic
groups act on a $6$-regular and a $10$-regular tree, respectively.
While the involutions $T_{6}=i+\varphi j$ and $T_{10}=i+j+k$ invert
an edge at the origin, the groups $\mathrm{Sym}_{3},\mathrm{Dih}_{5}\leq\mathrm{Alt}_{5}$
do not act transitively on its neighbors. 

\begin{rem}
Our focus throughout has been on optimal covering exponent and navigation.
If one relaxes the class number $1$ assumption or that $S$ consists
of a single prime, then one loses these optimal features. However
the gate sets that arise from such general $S$-arithmetic groups
coming from definite quaternion algebras still have reasonably good
covering exponents and navigation properties. These have been studied
in \cite{kliuchnikov2015framework}.
\end{rem}

\section{\label{sec:Iwahori-Hecke-operators}Iwahori-Hecke operators}

In this section we study asymmetric Hecke operators, both on $PU\left(2\right)$
and on the tree $X_{P}$, whose spectral properties coincides with
that of the non-backtracking random walk (NBRW) on a Ramanujan graph,
as explored in \cite{lubetzky2016cutoff}. This corresponds to words
(circuits) in free semigroups generated by Golden Gate sets.

\subsection{\label{subsec:Ramanujan-semigroups}Ramanujan semigroups in $PU\left(2\right)$}

For a Super-Golden-Gate set $\Gamma=\left\langle C,T\right\rangle $,
and an edge $e_{0}$ flipped by $T$ (as in Section \ref{sec:Super-Golden-Gates}),
denote by $\mathfrak{T}$ the ``sector'' descended from $e_{0}$,
namely, all vertices $v$ for which the shortest path from $v_{0}$
to $v$ begins with $e_{0}$. Let
\begin{equation}
S=\left\{ Tc\,\middle|\,1\neq c\in C\right\} \qquad\text{and}\qquad S^{r}=\left\{ s_{1}\cdot\ldots\cdot s_{r}\,\middle|\,s_{i}\in S\right\} ,\label{eq:S-def}
\end{equation}
and observe $\Sigma:=\bigcup_{r=0}^{\infty}S^{r}$, the semigroup
generated by $S$. Assuming that $e_{0}$ is leading away from the
origin, the map $e_{0}\mapsto\sigma e_{0}$ gives a correspondence
between $S^{r}$ and the edges which lead from the $r$-th to the
$\left(r+1\right)$-th level in $\mathfrak{T}$. In particular, $\Sigma$
is a free semigroup on $S$, and $\left|S^{r}\right|=k^{r}$, where
$k=\left|S\right|=\left|C\right|-1$. The spectrum of $T_{S}=T_{S,a_{s}}$
(with $a_{s}\equiv\frac{1}{k}$, in the notation of Section \ref{subsec:Spectral-gap-and-SA})
is particularly nice:
\begin{equation}
\Spec\left(T_{S}\big|_{L_{0}^{2}\left(PU\left(2\right)\right)}\right)\subseteq\left\{ \lambda\in\mathbb{C}\,\middle|\,\left|\lambda\right|=\frac{1}{\sqrt{k}}\right\} \cup\left\{ \frac{1}{k}\right\} .\label{eq:spec-T_S-sqrt}
\end{equation}
This should be compared with \cite[Thm.\ 1.3]{lubotzky1986hecke},
which asserts that for any \emph{symmetric} set $S'\subseteq PU\left(2\right)$
of size $k$ 
\[
\max\Spec\left(T_{S'}\big|_{L_{0}^{2}\left(PU\left(2\right)\right)}\right)\geq\frac{2\sqrt{k-1}}{k}.
\]
While the latter bound corresponds to the spectral radius of the random
walk on the Cayley graph of a free group, the bound $\frac{1}{\sqrt{k}}$
in (\ref{eq:spec-T_S-sqrt}) is the spectral radius of the random
walk on a free semigroup. Furthermore, since $T_{S^{r}}$ is precisely
$T_{S}^{r}$, we obtain $\left|\lambda\right|\leq k^{-r/2}$ for every
$\lambda\in\Spec\left(T_{S^{r}}\big|_{L_{0}^{2}\left(PU\left(2\right)\right)}\right)$. 

One should note, however, that $T_{S}$ is not normal, so that (\ref{eq:spec-T_S-sqrt})
is not enough to determine the spectral norm $W_{S}$. In fact, it
turns out that $W_{S}=1$. However, 
\begin{equation}
W_{S^{r}}=\sqrt{\frac{1}{2k^{r+1}}\left(\left(k-1\right)r\sqrt{r^{2}\left(k-1\right)^{2}+4k}+r^{2}\left(k-1\right)^{2}+2k\right)}=O\left(\frac{r}{k^{r/2}}\right).\label{eq:W_S-r}
\end{equation}
This, as well as (\ref{eq:spec-T_S-sqrt}), is obtained by considering
the action of $T_{S}$ on the tree $X_{p}$ and applying strong approximation
and the Ramanujan conjectures. We describe the case which corresponds
to the standard Hamilton quaternions for simplicity. Since $\Gamma$
acts simply transitively on the edges of $X_{p}$, we have 
\begin{equation}
L^{2}\left(PU\left(2\right)\right)\cong\Gamma\backslash\left(G_{\mathbb{R}}\times G_{p}\right)/B_{p},\label{eq:strong-app-iwahori}
\end{equation}
where $G=D^{\times}/Z\left(D^{\times}\right)$, and $B_{p}$ is the
Iwahori subgroup of $G_{p}:=G_{\mathbb{Q}_{p}}$, namely, $B_{p}=\mathrm{Stab}_{G_{p}}\left(e_{0}\right)$.
Via the isomorphism (\ref{eq:strong-app-iwahori}), $L^{2}\left(PU\left(2\right)\right)$
is a representation space for the Iwahori-Hecke algebra of $G_{p}$.
This is the algebra of compactly supported bi-$B_{p}$-invariant functions
on $G_{p}$, which acts on $L^{2}\left(G_{p}/B_{p}\right)\cong\mathrm{Edges}\left(X_{p}\right)$
by convolution, and $T_{S}$ acts by an element in this algebra.

The Ramanujan conjectures imply that every irreducible, infinite-dimensional
Iwahori-spherical representation which appears in $L^{2}\left(PU\left(2\right)\right)$
is tempered, so that the matrix coefficients of its Iwahori-fixed
vectors are in $L^{2+\varepsilon}\left(\mathrm{Edges}\left(X_{p}\right)\right)$
for all $\varepsilon>0$. The operator $T_{S}$ corresponds to NBRW
on $X_{p}$, and thus its spectrum is contained in that of the NBRW
on a $\left(k+1\right)$-regular tree, which is $\left\{ \smash{\lambda\in\mathbb{C}\,|\,\left|\lambda\right|=1/\sqrt{k}}\right\} \cup\left\{ \pm1/k\right\} $.
Furthermore, every tempered, Iwahori-spherical, unitary irreducible
representation $V$ of $G_{p}\cong PGL_{2}\left(\mathbb{Q}_{p}\right)$
is either ${\tt \left(1\right)}$ a twist of the Steinberg representation,
in which case $\dim V^{B_{p}}=1$, and $T_{S}\big|_{V^{B_{p}}}=\pm\frac{1}{k}$;
or ${\tt \left(2\right)}$ a principal series representation induced
from 
\[
\chi:\left(\begin{matrix}* & *\\
0 & *
\end{matrix}\right)\rightarrow\mathbb{C}^{\times},\quad\chi\left(\left(\begin{matrix}a & b\\
0 & d
\end{matrix}\right)\right)=\left(\frac{s}{\sqrt{p}}\right)^{\ord_{p}\left(a/d\right)}
\]
for some $s\in\mathbb{C}$ of norm $1$. In this case $\dim V^{B_{p}}=2$,
and in an appropriate orthonormal basis for $V^{B_{p}}$
\[
T_{S}\big|_{V^{B_{p}}}=\left(\begin{matrix}\frac{s}{\sqrt{p}} & 0\\
\frac{\left(p-1\right)s}{p} & \frac{1}{s\sqrt{p}}
\end{matrix}\right),
\]
from which (\ref{eq:W_S-r}) follows. We remark that this decomposition,
in the graph setting, is the key ingredient in \cite{lubetzky2016cutoff}.

\subsection{Ramanujan digraphs}

The golden-gate sets of \cite{lubotzky1986hecke,lubotzky1987hecke}
were used in \cite{LPS88} to construct explicit \emph{Ramanujan graphs},
which are finite $k$-regular graphs whose nontrivial spectrum is
contained within that of the $k$-regular tree. Namely, every eigenvalue
$\lambda$ of their adjacency operator is either trivial ($\lambda=\pm k$)
or satisfies $\left|\lambda\right|\leq2\sqrt{k-1}$. 

In a similar manner, the set $S$ defined in (\ref{eq:S-def}) can
be used to construct \emph{Ramanujan digraphs}; we say that a $k$-regular
digraph $\mathcal{G}$ is Ramanujan if every $\lambda\in\Spec\left(\mathrm{Adj}_{\mathcal{G}}\right)$
satisfies $\lambda=\left|k\right|$ or $\left|\lambda\right|\leq\sqrt{k}$.
In other words, its nontrivial spectrum is contained within that of
the Cayley digraph of the free semigroup on $k$ generators. For more
on Ramanujan digraphs, see \cite{Lubetzky2017RandomWalks}.

For $K,\mathcal{O},P,C_{k},T_{k}$ as in one of the examples in Section
\ref{subsec:Examples}, let $S_{k}=\left\{ T_{k}\cdot c\,\middle|\,1\neq c\in C_{k}\right\} $.
The Cayley graph of $\Gamma=\left\langle C_{k},T_{k}\right\rangle $
with respect to $S_{k}$ is naturally identified with the directed
line graph of the tree $X_{P}$. Namely, its vertices correspond to
directed edges in $X_{P}$, and its edges to non-backtracking steps
in $X_{P}$, so that the adjacency operator on $\mathrm{Cay}\left(\Gamma,S_{k}\right)$
describes NBRW on a tree. For an ideal $Q\neq P$ in $\mathcal{O}_{K}$,
let $\Gamma_{Q}=\ker\left(\Gamma\rightarrow G_{\mathcal{O}_{K}/Q}\right)$,
and let $Y^{P,Q}$ be the Cayley graph of $\Gamma/\Gamma_{Q}$ with
respect to the set of generators $S_{k}$. Again, $L^{2}\left(\mathrm{Verts}\left(Y^{P,Q}\right)\right)$
is a representation of the Iwahori-Hecke algebra of $G_{P}$, and
the Ramanujan conjectures imply that every nontrivial eigenvalue $\lambda$
of $\mathrm{Adj}\left(Y^{P,Q}\right)$ corresponds to a Iwahori-fixed
vector in a tempered representation, giving $\lambda=\pm1$ or $\left|\lambda\right|=\sqrt{k}$,
as in Section \ref{subsec:Ramanujan-semigroups}. In addition, we
have the trivial eigenvalue $k$ (which arises from the trivial representation),
but unlike the continuous case here we also obtain the eigenvalue
$-k$ whenever $P$ is a quadratic non-residue in $\nicefrac{\mathcal{O}_{K}}{Q}$.

Finally, we note that the groups $\Gamma/\Gamma_{Q}$ can be identified
concretely: denoting $\mathbb{F}_{q}=\mathcal{O}_{K}/Q$, we have
by Wedderburn's theorem 
\[
G_{\mathbb{F}_{q}}=\left(D\otimes\mathbb{F}_{q}\right)^{\times}/\mathbb{F}_{q}^{\times}\cong\mathrm{Mat}_{2}\left(\mathbb{F}_{q}\right)^{\times}/\mathbb{F}_{q}^{\times}=PGL_{2}\left(\mathbb{F}_{q}\right),
\]
and the image of $\Gamma/\Gamma_{Q}$ in $G_{\mathbb{F}_{q}}$ is
either $PGL_{2}\left(\mathbb{F}_{q}\right)$ or $PSL_{2}\left(\mathbb{F}_{q}\right)$,
according to the Legendre symbol $\left(\frac{P}{Q}\right)$. For
example, taking $C_{4}=\left\{ i,j,k\right\} $, $T_{4}=i+j+k$ (see
Section \ref{subsec:Lipschitz-quaternions}) and $Q=\left(23\right)$
one obtains $\Gamma/\Gamma_{Q}\cong PSL_{2}\left(\mathbb{F}_{23}\right)$
by $i\mapsto\left(\begin{smallmatrix}0 & 11\\
2 & 0
\end{smallmatrix}\right)$, $j\mapsto\left(\begin{smallmatrix}15 & 1\\
4 & 8
\end{smallmatrix}\right)$. This gives $T_{4}\cdot i\mapsto\left(\begin{smallmatrix}16 & 5\\
20 & 5
\end{smallmatrix}\right)$, and similarly for $T_{4}\cdot j,T_{4}\cdot k$. The spectrum of
the Ramanujan digraph obtained is shown in Figure \ref{fig:PSL_2_23}.

\newsavebox{\smlmat} 
\savebox{\smlmat}{$\vphantom{\Big|}S=\left\{\ensuremath{\left(\begin{smallmatrix} 16 & 5\\ 20 & 5 \end{smallmatrix}\right),\left(\begin{smallmatrix} 20 & 8\\ 5 & 1 \end{smallmatrix}\right),\left(\begin{smallmatrix} 7 & 10\\ 21 & 14 \end{smallmatrix}\right)\right\}}$}

\begin{figure}[h]
\captionsetup{format=plain, margin=0.5em, justification=raggedright}
\floatbox[{\capbeside}]{figure} {\caption{The adjacency spectrum of $\overset{\ }{Y^{3,23}}$, the Cayley graph of $PSL_{2}\left(\mathbb{F}_{23}\right)$ w.r.t.\ \usebox{\smlmat}.}\label{fig:PSL_2_23}} {\includegraphics[scale=0.33]{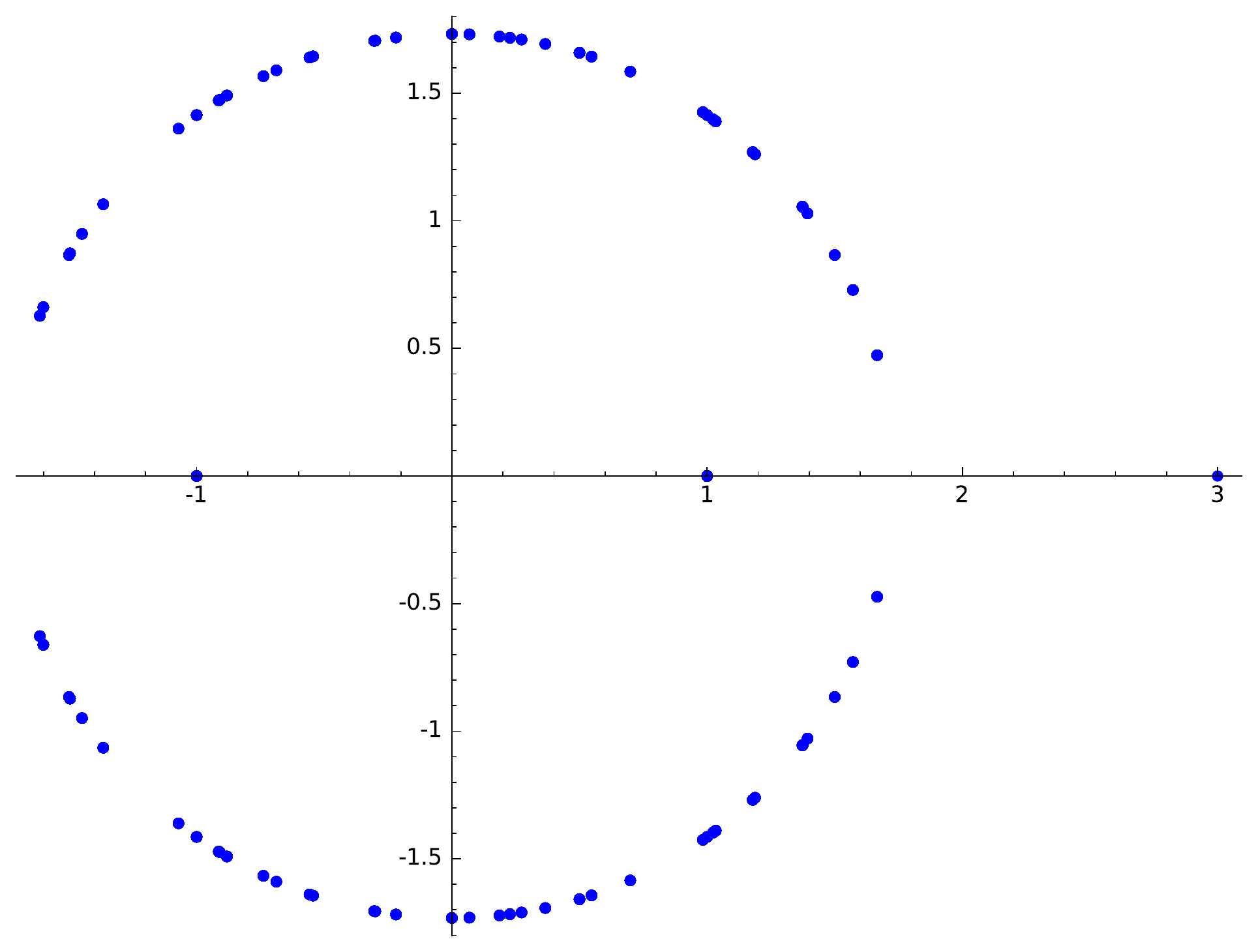}}
\end{figure}

\bibliographystyle{alpha}
\bibliography{/home/ori/Math/mybib}

\end{document}